\documentclass[12pt]{article}
\usepackage[left=1in,right=1in,top=1in,bottom=1in,letterpaper]{geometry}

\usepackage{amsopn}
\usepackage{amsfonts}
\usepackage{hyperref}
\usepackage{graphicx}
\usepackage{epstopdf}
\usepackage{algpseudocode}
\usepackage{algorithm}
\usepackage{multirow}
\usepackage{amsthm}

\usepackage{graphicx}        
\usepackage{multicol}        

\usepackage{newtxtext}       %
\usepackage[varvw]{newtxmath}       

\graphicspath{{./pic/}}

\def\x{{\mathbf x}}

\def\A{{\mathcal A}}
\def\C{{\mathcal C}}
\def\B{{\mathcal B}}
\def\D{{\mathcal D}}
\def\J{{\mathcal J}}

\def\X{{\mathcal X}}
\def\Z{{\mathcal Z}}
\def\U{{\mathcal U}}
\def\V{{\mathcal V}}
\def\S{{\mathcal S}}
\def\Y{{\mathcal Y}}

\def\W{{\mathcal W}}

\newcommand{\norm}[1]{\left\lVert#1\right\rVert}
\newcommand{\R}{\mathbb{R}}

\DeclareMathOperator*{\sign}{sign}
\DeclareMathOperator{\prox}{prox}

\DeclareMathOperator*{\argmin}{argmin}

\newtheorem{theorem}{Theorem}[section]
\newtheorem{lemma}[theorem]{Lemma}
 \newtheorem{definition}[theorem]{Definition}
 \newtheorem{proposition}[theorem]{Proposition}
 \newtheorem{corollary}[theorem]{Corollary}

 \newtheorem{remark}[theorem]{Remark}

\begin{document}

\title{Power of $\ell_1$-Norm Regularized Kaczmarz Algorithms for High-Order Tensor Recovery}
\author{Katherine Henneberger\thanks{Department of Mathematics, University of Kentucky, Lexington, KY 40511}
\and Jing Qin\thanks{Department of Mathematics, University of Kentucky, Lexington, KY 40511}}
\date{}

\maketitle

\abstract{Tensors serve as a crucial tool in the representation and analysis of complex, multi-dimensional data. As data volumes continue to expand, there is an increasing demand for developing optimization algorithms that can directly operate on tensors to deliver fast and effective computations. Many problems in real-world applications can be formulated as the task of recovering high-order tensors characterized by sparse and/or low-rank structures. In this work, we propose  novel Kaczmarz algorithms with a power of the $\ell_1$-norm regularization for reconstructing high-order tensors by exploiting sparsity and/or low-rankness of tensor data. In addition, we develop both a block and an accelerated variant, along with a thorough convergence analysis of these algorithms. A variety of numerical experiments on both synthetic and real-world datasets demonstrate the effectiveness and significant potential of the proposed methods in image and video processing tasks, such as image sequence destriping and video deconvolution.}

\section{Introduction}

The Kaczmarz algorithm \cite{karczmarz1937angenaherte} is an iterative method for solving linear systems of equations, and is particularly beneficial for large and sparse systems. This method estimates a solution to a consistent linear system by iterative projection onto the hyperplane corresponding to each equation. The efficiency of the Kaczmarz algorithm is widely acknowledged across various fields, including computational tomography \cite{gordon1970algebraic, popa2004kaczmarz}, brain imaging \cite{zhou2013tensor} and image processing \cite{eggermont1981iterative,chen2021regularized}, underscoring its utility in practical scenarios. Building upon the classical Kaczmarz algorithm, several extensions have been developed, including those that incorporate randomness, block implementation, and regularization. In particular, the Randomized Kaczmarz Algorithm (RKA) \cite{strohmer2009randomized} introduces a random scheme of selecting projections at each step, aligning with methods like coordinate descent or stochastic gradient descent under certain conditions. To improve the convergence rate, the block Kaczmarz method \cite{needell2014paved,needell2015randomized} has been proposed to group indices into blocks for each iteration and thereby consider multiple linear constraints simultaneously, which has shown the computational effectiveness in handling extensive linear systems. Recently, regularized Kaczmarz algorithms have been developed to enhance stability and convergence of the classical Kaczmarz algorithm for solving ill-conditioned linear systems by incorporating additional regularization terms that account for the sparsity or smoothness of the solution \cite{burger2006regularizing,lorenz2014linearized,chen2021regularized}. While many Kaczmarz type algorithms have been designed to expedite computations by reducing the desired number of iterations, they possess limited capability to manage large-scale datasets with multi-linear structures.

Recently, to represent the multi-linear structure and intricate patterns in data, tensor modeling and optimization have become prevalent in various applications. In particular, Kaczmarz type algorithms have been extended to solve tensor linear systems by adapting the method for related tensor recovery problems. For example, the Kaczmarz algorithm can be adapted from matrix-vector to tensor-vector multiplication scenarios \cite{wang2023solving}, as well as third-order tensor-tensor products \cite{chen2021regularized,ma2022randomized}. The recent literature on Kaczmarz-type tensor algorithms predominantly deals with third-order tensors. Extending these algorithms to higher-order tensors presents a significant challenge, as the conventional t-product and its algebraic frameworks are defined solely for third-order tensors. Yet, higher-order tensors frequently occur in practical scenarios, such as fourth-order color videos and hyperspectral data. Additionally, these tensors find application in areas like image inpainting, sparse signal recovery, and image destriping. Operating directly with tensors, as opposed to converting them into matrices, is beneficial because it maintains the multidimensional structure of the data and underlying correlations.

Tensor Kaczmarz algorithms can be further extended to incorporate regularization terms or to handle tensor linear constraints, similar to their matrix counterparts. These extensions aim to improve the convergence and stability of the algorithm, especially for large-scale and/or ill-conditioned tensor systems. These methods are particularly effective in tensor recovery problems, where the goal is to estimate tensors from incomplete or corrupted data. Common assumptions in these problems include the sparsity or low-rankness of the tensor, reflecting the underlying regularity of the data. In tensor recovery models, the regularization term plays a critical role in encouraging sparse or low-rank solutions. There are many convex regularizations in sparse signal recovery that can preserve sparsity of the solution, such as the $\ell_1$-norm \cite{tibshirani1996regression} and log-sum penalty \cite{candes2008enhancing}. Recent advancements demonstrate how Kaczmarz algorithms integrated with regularization techniques, such as the $\ell_1$-norm \cite{chen2021regularized} and log-sum penalty \cite{henneberger2023log}, deliver promising outcomes for addressing tensor recovery challenges. Unlike the more common $\ell_1$-norm, which promotes sparsity uniformly by mapping small elements to zero, the $p$-th power of the $\ell_1$-norm, also denoted by $\ell_1^p$, introduces a threshold that depends on the relationship between the components of a vector when $p>1$ \cite{prater2023constructive}. In addition, the availability of a closed-form proximal operator for this norm simplifies the implementation of iterative algorithms, making it a compelling choice for promoting relative sparsity in complex optimization scenarios. This background sets the stage for our investigation, where we improve upon recent regularized algorithms as we explore tensor recovery for tensors of order three or higher, employing a power function of the $\ell_1$-norm regularization term to tackle sparse or low-rank tensor recovery challenges efficiently.

In this paper, motivated by the advantages of the power of the $\ell_1$ norm regularization and building upon the existing Kaczmarz extensions, we introduce a novel  regularized Kaczmarz algorithm framework for reconstructing high-order tensors from partial tensor data. Comprehensive convergence analysis for various algorithmic variants is provided. Through extensive numerical experiments on various tensor data sets, we validate the convergence properties and demonstrate the effectiveness of the proposed algorithms in addressing high-order tensor recovery tasks.

The major contributions of this paper are as follows:
\begin{enumerate}
    \item We derive a novel closed form of the proximal operator of the power of $\ell_1$ norm is proposed, together with useful properties. Based on this, we propose two novel regularized Kaczmarz algorithms for recovering high-order tensors with sparse or low-rank structures, and analyze the convergence of these algorithms.
    \item To further expedite computation and convergence, we present two extensions of the new algorithm framework: its block and accelerated variants accompanied with convergence guarantees.
    \item Numerical experiments on diverse synthetic and real world data justify the effectiveness of the proposed methods. Furthermore, applications to image destriping and color video deblurring demonstrate the utility of the proposed methods, together with practical guidance about parameter selection and computational complexity analysis.
\end{enumerate}
The rest of the paper is organized as follows: Section~\ref{sec:Prelim} introduces basic concepts and results for high-order tensors and convex optimization. Our regularized Kaczmarz algorithms are proposed in Section~\ref{sec:Prop} along with thorough convergence analysis in Section~\ref{sec:Converg}. Extensions of our algorithm, including a block variant and an accelerated version, are detailed in Section~\ref{sec:exten}, and various numerical experiments are shown in Section~\ref{sec:num}. Finally, the conclusion and future work are summarized in Section~\ref{sec:Con}.

\section{Preliminaries}\label{sec:Prelim}

Throughout this paper, vectors are represented by bold lowercase letters, e.g., $\x$; matrices are represented by capital letters, e.g., $X$;  tensors are denoted by calligraphic letters, e.g., $\mathcal{X}$. We use $\|A\|_*=\sum_{i}\sigma_i(A)$ to denote the matrix nuclear norm. Let $\R$ be the set of all real numbers, $\mathbb{C}$ the set of all complex numbers, and $[n]=\{1,2,\ldots,n\}$ where $n$ is an integer. We denote the set of all $m$-th order real or complex tensors of size $n_1\times \ldots\times n_m$ by $\R^{n_1\times \ldots \times n_m}$ or $\mathbb{C}^{n_1\times \ldots\times n_m}$. To make the paper self-contained, we follow the convention in \cite{qin2022low}, and list the most fundamental definitions related to tensor forms and operations in Table~\ref{tab:definitions}, as well as in Definitions~\ref{def:tensors}, \ref{def:orth}, \ref{def:fdiag}, and \ref{def:tsvd}.

\begin{table}[htp]
\centering
\begin{tabular}{l p{10cm}}
\hline
\textbf{Notation} & \textbf{Explanation} \\
\hline
$\A\odot \B$ & Hadamard product of two tensors $\A,\B$ with the same size, i.e., $(\A\odot \B)(i_1,i_2,...,i_m)=\A(i_1,i_2,...,i_m)\odot \B(i_1,i_2,...,i_m)$\\
$\norm{\A}_1$ & tensor $\ell_1$ norm, i.e., the sum of the absolute values of all the entries in $\A$\\
 $\A_i$ &$\A_i = \A(:,...,:,i)$  the order-$(m-1)$ tensor slice constructed by keeping the $m$-th index of $\A$ fixed at $i$\\

$\text{circ}(\mathcal{A})$ & $
        \text{circ}(\A) = \begin{pmatrix}
        \A_1 & \A_{n_m} & \A_{n_m}-1& \cdots & \A_2\\
        \A_2 & \A_{1} & \A_{n_m}& \cdots & \A_3\\
        \vdots & \vdots & \vdots & \ddots& \vdots\\
         \A_{n_m} & \A_{n_m}-1& \A_{n_m}-2& \cdots & \A_1
        \end{pmatrix}
   $ \\

$\text{circ}^k(\mathcal{A})$ & repeated application of the operation $\text{circ}(\cdot)$ to $\mathcal{A}$ by $k$ times, where $\text{circ}^{-1}$ is the inverse operation \\

$\text{unfold}(\mathcal{A})$ & $\text{unfold}(\A) = \begin{pmatrix}
        \A_1\\
        \A_2 \\
        \vdots\\
         \A_{n_m} \end{pmatrix}$ \\

$\text{fold}(A)$ & inverse operation of $\text{unfold}(\cdot)$, folding the matrix back into a tensor. \\

$\A^j$& ${\A^j=\A(:,:,i_3,...,i_m)}$  where $j = (i_m-1)n_3\cdots n_{m-1}+\cdots+(i_4-1)n_3+i_3$\\

$\text{bdiag}(\mathcal{A})$&   block diagonal matrix form of $\A$  with diagonal blocks $\A^1,\ldots, \A^J$ where $J=n_3n_4\cdots n_m$ \\

$\mathcal{A}\Delta\mathcal{B}$ & facewise product of $\A$ and $\B$, denoted by $\C=\A\Delta\B$, is defined through
         $\text{bdiag}(\C)=\text{bdiag}(\A)\cdot\text{bdiag}(\B)$\\

$\A*_L\B$ & $L$-based t-product $\A*_L\B = L^{-1}(\A_L\Delta \B_L)$\\

$\mathcal{A}_L^*$ & complex conjugate transpose of the transformed tensor $\A_L$, defined as
${(\A^*)_L(:,:,i_3,...,i_m) = (\A_L(:,:,i_3,...,i_m))^*}$
for all ${i_j\in[n_j]}$, ${j\in\{3,...,m\}}$. \\

$\J$ & identity tensor if $\J_L(:,:,i_3,...,i_m)=I_n$ for all ${i_j\in [n_j]}$, ${j\in\{3,...,m\}}$.\\
\hline
\end{tabular}
\caption{Summary of main notation and operations. Here ${\A\in\R^{n_1\times n_2 \times n_3\times ...\times n_m}}$ and ${\B\in\R^{n_2\times l\times n_3\times ...\times n_m}}$.}
\label{tab:definitions}
\end{table}

\begin{definition}\cite{qin2022low}\label{def:tensors}

 A general \textbf{linear transform}
        $L:\mathbb{C}^{n_1\times \ldots\times n_m}\to \mathbb{C}^{n_1\times \ldots\times n_m}$ is defined as $$ L(\A) = \A\times_3 U_{n_3}\times ...\times_m U_{n_m}:=\A_L,$$ where $U_{n_i}\in{\mathbb{C}^{n_i\times n_i}}$ are the corresponding transform matrices. When the matrices $\{U_{n_i}\}_{i=3}^m$ are invertible, the linear transform $L$ is called invertible. Moreover, we assume there exists a constant $\rho>0$ such that
        the corresponding matrices $\{U_{n_i}\}_{i=3}^m$ of the invertible linear transform $L$ satisfy the following condition:
        \[
        \begin{aligned}
            &(U_{n_m}\otimes U_{n_{m-1}}\otimes ...\otimes U_{n_3})\cdot(U^*_{n_m}\otimes U^*_{n_{m-1}}\otimes ...\otimes U^*_{n_3})\\ &=(U^*_{n_m}\otimes U^*_{n_{m-1}}\otimes ...\otimes U^*_{n_3}) \cdot(U_{n_m}\otimes U_{n_{m-1}}\otimes ...\otimes U_{n_3})\\
        &=\rho I,
        \end{aligned}\]
        where $\otimes$ denotes the Kronecker product and $U_{n_i}^*$ is the complex conjugate transpose of $U_{n_i}$.
        Here the identity matrix $I$ is of size $n_3 n_4....n_m \times n_3 n_4....n_m $.
\end{definition}

There are many ways to define the linear transform $L$. For example, if each $U_{n_i}$ is the matrix of the 1D Fourier transform, then the corresponding $L$ is called the \textbf{Fourier transform}, which can be implemented using the command ``\verb"fft"'' in Matlab which leads to $\rho = n_3 n_4\ldots n_m$. In general, if each $U_{n_i}$ is unitary, then $\rho=1$. 
For example, in the discrete cosine transform (DCT), each $U_{n_i}$ is an orthogonal DCT matrix. Based on the linear transform $L$, we can define the $L$-based t-product (see Table~\ref{tab:definitions}). It is worth noting that the third-order case of the $L$-based t-product reduces to the matrix-mimetic tensor-tensor product introduced in \cite{kilmer2019tensor}.

\begin{definition}\cite{qin2022low}\label{def:orth}
A tensor $\mathcal{X}$ is \textbf{orthogonal }if $\mathcal{X}^**_L\mathcal{X}=\mathcal{X}*_L\mathcal{X}^*=\J$.
\end{definition}

\begin{definition}\cite{qin2022low} \label{def:fdiag}
A tensor $\X$ is \textbf{f-diagonal} if each frontal slice ${\X(:,:,i_3,...,i_m)}$ is diagonal for all  $i_j\in [n_j]$ and $j\in\{3,...,m\}$.
\end{definition}
\begin{definition}\cite{qin2022low}\label{def:tsvd}
The \textbf{tensor Singular Value Decomposition (t-SVD)} of a tensor $\X\in\R^{n_1\times n_2\times\ldots\times n_m}$ induced by $*_L$ is
\[
\mathcal{X}=\mathcal{U}*_L\mathcal{S}*_L\mathcal{V}^*\] where $\mathcal{U}\in\R^{n_1\times n_1\times n_3\times\ldots\times n_m}$ and $\mathcal{V}\in\R^{n_2\times n_2\times n_3\times \ldots\times n_m}$ are orthogonal and the core tensor $\mathcal{S}\in\R^{n_1\times n_2\times n_3\times\ldots\times n_m}$ is f-diagonal.
Moreover, the \textbf{t-SVD rank} of tensor $\X$ is defined as \[\text{rank}_{\text{t-SVD}}(\X) = \#\{i: S(i,i,1,1,...,1)\neq \mathcal{O}\}\]
where $\#$ denotes the cardinality of a set and $\mathcal{O}$ is the zero tensor of size $n_3\times\ldots\times n_m$.
\end{definition}

\begin{corollary}\label{cor:refsep}
There exist several separability properties about the $L$-based t-product. Consider $\A\in\R^{n_1\times n_2 \times n_3\times ...\times n_m}$ and $\X\in\R^{n_2\times l\times n_3\times ...\times n_m}$.
\begin{enumerate}
\item Separability in the first dimension.
Let $\A(i) = \A(i,:,:,\ldots,:)$ denote the $i$th horizontal slice of size $1\times n_2\times n_3\times...\times n_m$. Then,
\[ (\A *_L \X) (i) = L^{-1}(\A_L(i)\Delta \X_L).\]
\item Sum separability in the second dimension.
\begin{align}
    (\A *_L \X) = \sum_{j=1}^{n_2} \A(:,j,:,...,:)*_L\X(j,:,:,...,:).\label{sumsep}
\end{align}
\item Circular convolution in higher dimensions. For $\A,\X\in\R^{1\times...\times 1\times n_j\times 1...\times 1}$ when $j = 2,\ldots,m$, we have
\[\text{squeeze}(\A*_L\X)=\text{circ(squeeze}(\A))\text{squeeze}(\X),\]
where $\text{squeeze}(\cdot)$ removes the dimensions of size one from the given tensor. In particular, if $j=m=3$, then this property reduces to the special case in \cite{chen2021regularized}.

\item Frontal slice separability.
For each frontal slice $j$, we have
\[(L(\A*_L\X))^j = (L(\A))^j(L(\X))^j\]
In addition, if one of the high dimensions from $\{i_3,...,i_m\}$ is fixed as one,  we have
\[
(L(\A*_L\X))(:,:,...,1,...) = (L(\A))(:,:,...,1,...)\Delta(L(\X))(:,:,...,1,...).
\]

\item Mode-wise unfolding separability.
Let $U_{n_i}\in{\C^{n_i\times n_i}}$ be the corresponding transform matrices to $L:\R^{n_1\times n_2\times ....\times n_m}\rightarrow \R^{n_1\times n_2\times ....\times n_m} $ such that $\A_L = L(\A) = \A\times_3 U_{n_3}\times_4 ...\times_m U_{n_m}$. By letting $\A_{(j)}$ be the mode-$j$ matrix unfolding of tensor $\A$, then we have
\begin{align*}
    (\A*_L\X)_{(j)}&= U^{-1}_{n_j}(\A_L\Delta \X_L)_{(j)}(U^{-1}_{n_m}\otimes U^{-1}_{n_{m-1}}\otimes \cdots \otimes U^{-1}_{n_3}\otimes I_{n_2}\otimes I_{n_1})^*.
\end{align*}
\item Let $\A\in\R^{n_1\times n_2\times n_3...\times n_m}$ and $\X\in\R^{n_2\times l\times n_3...\times n_m}$. Then the classical t-product, i.e., the $L$-based t-product when $L$ is the identity map, is defined as
\begin{align}\label{convolution}   \underbrace{\A*\X}_{\mbox{order-$m$ product}} = \text{fold}(\underbrace{\text{circ}(
    \A)*\text{unfold}(\X)}_{\mbox{order-$(m-1)$ product}}).
\end{align}
By applying this property repeatedly, we get
\begin{align*}
    \A*\X &= \sum_{j=1}^{n_2} \A(:,j,:,...,:)*\X(j,:,:,...,:)& \text{by \eqref{sumsep}}\\
    &= \sum_{j=1}^{n_2} \text{fold}(\text{circ}(\A(:,j,:,...,:))*\text{unfold}(\X(j,:,:,...,:)))\\
     &= \sum_{j=1}^{n_2} \text{fold}(\text{circ}^2(\A(:,j,:,...,:))*\text{unfold}^2(\X(j,:,:,...,:)))\\
    &\phantom{1} \vdots\\
     &= \sum_{j=1}^{n_2} \text{fold}(\text{circ}^{m-2}(\A(:,j,:,...,:))
    \text{unfold}^{m-2}(\X(j,:,:,...,:)))\\
\end{align*} where $\text{circ}^{m-2}(\A(:,j,:,...,:)$ is the block circulant matrix at the base level of the $\text{circ}$ operator and is of size $n_1n_3...n_m\times n_3...n_d$. In addition, $\text{unfold}^{m-2}(\X(j,:,:,...,:)$ is the $n_3...n_d\times l$ matrix formed by applying $\text{unfold}(\cdot)$ repeatedly. Note that when $\A\in\R^{1\times n_2\times n_3}$ and $\X\in n_2\times 1\times n_3$ this reduces to the sum of vector convolutions \cite[(2.21)]{chen2021regularized}.

    \end{enumerate}
\end{corollary}

\begin{remark}
    In \cite{qin2022low} and \cite{lu2018tensor}, the algorithm used to calculate the complex conjugate transpose acts on each frontal slice of $\A$. Thus, for any tensor $\A$, we have
\begin{align}\label{bdiag}
    (\text{bdiag}(\A))^* = \text{bdiag}(\A^*).
\end{align}
By the fact that the complex conjugate transpose is calculated on the frontal slices, we have
\begin{align}
    (\A_L)^*(:,:,i_3,...,i_m) = (\A_L(:,:,i_3,...,i_m))^* = (\A^*)_L(:,:,i_3,...,i_m).\label{property3}
\end{align}
\end{remark}

\begin{definition}\label{def:stconv}
A continuous function $f: \R^{n_1\times n_2\times ...\times n_m}\rightarrow \R$ is called \textbf{$\alpha$-strongly convex} for some $\alpha>0$ if
it satisfies
 \[
 f(\Y)\geq f(\X)+\langle\X^*,\Y-\X\rangle +\frac{\alpha}{2}\|\Y-\X\|_F^2,
 \]
for any $\X,\Y\in \R^{n_1\times n_2\times ...\times n_m}$ and $\X^*\in\partial f(\X)$.
\end{definition}

\section{Proposed Algorithms}\label{sec:Prop}
Consider a tensor recovery problem where only a subset of the entries of the high-order tensor are available or are damaged by some additive noise. Characteristics of the underlying tensor such as sparsity in some transform domain or by itself, or low-rankness with respect to a specific tensor decomposition, usually facilitate the solution estimation.
In this section, we propose two novel regularized Kaczmarz algorithms: one for sparse tensor recovery and the other for low-rank tensor recovery. Both algorithms incorporate the power of the $\ell_1^p$ norm regularization.

\subsection{Proximal Operator of \texorpdfstring{$\ell_1^p$}{L1p}}\label{subsec:prox}

Proximal operators have been widely used in convex optimization algorithms especially when the objective function is non-differentiable \cite{parikh2014proximal}. In this work, we focus on the \textbf{proximal operator} of  $\ell_1^p$ in the tensor setting which is defined as
\begin{align}\label{eqn:prox}
    \prox_{\lambda \norm{\cdot}_1^p}(\Z) = \argmin_{\X\in\R^{n_2\times l\times\ldots\times n_m}}\frac{1}{2}\|\X-\Z\|_F^2+ \lambda \|\X\|_1^p.
\end{align}
In what follows, we derive a novel closed-form solution of the proximal operator for the $\ell_1^p$ norm when $p=1,2,3,4$ for tensors which will be used in the proposed regularized Kaczmarz algorithms. Note that our derivations and solution forms are different from those by the iterative constructions in \cite{prater2023constructive}.

\begin{lemma}\label{lemma:prox}
If $\Z\in\R^{n_2\times l\times \ldots\times n_m}$ and $\lambda>0$,  the optimal solution $\X^*$ to \eqref{eqn:prox} when $p=1,2,3,4$ has the following form
\begin{equation}
    \X^*=\sign(\Z)\odot\max\{|\Z|-p\lambda g_p^{p-1},\mathcal{O}\}.
\end{equation}
The function $g_p(\norm{\Z}_1)$ depends on $p$ and is given by
\begin{equation}\label{eqn:gp}
g_p=\left\{
\begin{aligned}
&1,&&p=1;\\
&\frac{\norm{\Z}_1}{2n\lambda+1},&&p=2;\\
&\frac{-1+\sqrt{1+12n\lambda\norm{\Z}_1}}{6n\lambda},&&p=3;\\
&\sqrt[3]{\frac1{8n\lambda}(\norm{\Z}_1+\sqrt{\Delta})}+\sqrt[3]{\frac1{8n\lambda}(\norm{\Z}_1-\sqrt{\Delta})},&&p=4.
\end{aligned}
\right.
\end{equation}
Here $\Delta=\norm{\Z}_1^2+\frac1{27n\lambda}$ and $n=n_2ln_3\ldots n_m$.

\end{lemma}

\begin{proof}
When $p=1$, the problem \eqref{eqn:prox} boils down to the proximal operator of the $\ell_1$-norm, which can be directly expressed by the soft-thresholding operator.

When $p\geq2$, we first show that $\sign(\X) = \sign(\Z)$. The optimality condition of \eqref{eqn:prox} yields
\[
    \X-\Z + p\lambda \sign(\X)\|\X\|_1^{p-1} = \mathcal{O}.
\]
Due to the fact that $\X=\sign(\X)\odot |\X|$, we rewrite the above equation as
\[
\sign(\X)\odot\left(|\X| + p\lambda \|\X\|_1^{p-1}\right) = \sign(\Z)\odot|\Z|,
\]
Let $x_i$ and $z_i$ be the $i$-th entry of the respective vectorized version of $\X$ and $\Z$, and let $n=n_2ln_3\ldots n_m$. Then for $i=1,\ldots,n$, we have
\[
\sign(x_i)(|x_i|+p\lambda\norm{\X}_1^{p-1})=\sign(z_i)|z_i|.
\]
Note that $|x_i| + p\lambda \|\X\|_1^{p-1}$ is nonnegative and zero if and only if $\X=\mathcal{O}$. Thus $\sign(x_i) = \sign(z_i)$ and
\[
|x_i|+p\lambda \|\X\|_1^{p-1} = |z_i|.
\]
That is, $\sign(\X)=\sign(\Z)$ and
\begin{equation}\label{eqn:XZ}
|\X|+p\lambda\norm{\X}_1^{p-1}=|\Z|.
\end{equation}
By taking the $\ell_1$-norm on both sides of \eqref{eqn:XZ} and simplification, we obtain an equation about $\norm{\X}_1$ of the form
\begin{equation}\label{eqn:X}
   pn\lambda\norm{\X}_1^{p-1}+\norm{\X}_1-\norm{\Z}_1=0.
\end{equation}
When $p=2,3,4$, we are able to find the real roots,and the nonnegative solution is denoted by $g_p$. In particular, if $p=2$, \eqref{eqn:X} becomes a linear equation which yields the solution
\[
\norm{\X}_1=\frac{\norm{\Z}_1}{2n\lambda+1}:=g_2(\norm{\Z}_1).
\]
If $p=3$, \eqref{eqn:X} becomes a quadratic equation and we obtain the nonnegative solution
\[
\norm{\X}_1=\frac{-1+\sqrt{1+12n\lambda\norm{\Z}_1}}{6n\lambda}:=g_3(\norm{\Z}_1).
\]
If $p=4$, \eqref{eqn:X} is a depressed cubic equation which has a real solution by applying the Cardano's formula of the form
\[
\norm{\X}_1=\sqrt[3]{\frac1{8n\lambda}(\norm{\Z}_1+\sqrt{\Delta})}+\sqrt[3]{\frac1{8n\lambda}(\norm{\Z}_1-\sqrt{\Delta})}:=g_4(\norm{\Z}_1),
\]
where the scaled determinant of the equation $\Delta=\norm{\Z}_1^2+\frac1{27n\lambda}$. Finally,
substitution $g_p$ into \eqref{eqn:XZ} completes the derivation of the proximal operator.

\end{proof}

Note that if $p\geq5$, then there is no closed-form solution in general for the proximal operator of $\ell_1^p$ as it involves finding a root of a polynomial with degree at least four.

\begin{lemma}\label{lem:sym}
Let $\Z\in\R^{n_1\times \ldots\times n_m}$ be a nonzero tensor and $\X^*=\prox_{\lambda\norm{\cdot}_1^p}(\Z)$. Assume $x^*_i$ and $z_i$ are the $i$-th entry of the vectorized version of the optimal solution $\X^*$ and $\Z$, respectively. Then for $i=1,\ldots,n$, the following properties about the sign and the order relationships between $\X^*$ and $\Z$ hold
\begin{equation}\label{eqn:proxprop1}
    x_i^*\begin{cases}
        \geq 0& \text{if }z_i>0\\
        \leq 0& \text{if }z_i<0\\
    \end{cases},
\end{equation}
and
\begin{equation}\label{eqn:proxprop2}
|x_i^*|\geq|x_j^*| \,\text{ if } \,|z_i|>|z_j|.
\end{equation}
In addition, the proximal operator is nonexpansive in the Frobenius norm
\begin{equation}\label{eqn:nonexp}
\|\prox_{\lambda\|\cdot\|_1^p}(\Z)\|_F\leq\|\Z\|_F.
\end{equation}
\end{lemma}
\begin{proof}
The relationships \eqref{eqn:proxprop1} and \eqref{eqn:proxprop2} can be verified by the fact that a smaller objective value could be obtained by changing the sign of $x_i$ or interchanging the absolute values $|x_i^*|$ and $|x_j^*|$.

Regarding \eqref{eqn:nonexp}, we first notice that $\prox_{\lambda\|\cdot\|_1^p}(\mathcal{O}) =\mathcal{O}$. Then we apply the nonexpansiveness of the general proximal operator \cite{rockafellar2015convex} and get
\begin{equation}
\|\prox_{\lambda\|\cdot\|_1^p}(\Z)\|_F = \|\prox_{\lambda\|\cdot\|_1^p}(\Z)-\prox_{\lambda\|\cdot\|_1^p}(\mathcal{O})\|_F \leq\|\Z\|_F.
\end{equation}

\end{proof}

\begin{lemma}
    Let $g(\X)=\lambda\norm{\X}_1^p+\frac12\norm{\X-\Z}_F^2$ and $\X^*=\prox_{\lambda\norm{\cdot}_1^p}(\Z)$. Then  for any tensor $\X$ with the same size as $\X^*$, we have
    \begin{equation}\label{eqn:g_decay}
    g(\X^*)-g(\X)\leq -\frac1{2}\norm{\X^*-\X}_F^2.
    \end{equation}
\end{lemma}

\begin{proof}
First, the optimality condition of $\X^*$ yields
\[
\mathcal{O}\in \partial(\lambda\norm{\X^*}_1^p)+(\X^*-\Z),
\]
which leads to $\Z-\X^*\in\partial(\lambda \norm{\X^*}_1^p)$. For any tensor $\X$, we have
\[\begin{aligned}
g(\X^*)-g(\X)&=\lambda\norm{\X^*}_1^p-\lambda\norm{\X}_1^p+\frac12\norm{\X^*-\Z}_F^2-\frac12\norm{\X-\Z}_F^2\\
&\leq \langle \Z-\X^*,\X^*-\X\rangle +\frac12\norm{\X^*-\X}_F^2+\langle \X^*-\X,\X-\Z\rangle\\
&=-\frac12\norm{\X^*-\X}_F^2.
\end{aligned}\]
Here the second inequality is guaranteed by $\Z-\X^*\in\partial(\lambda\norm{\X^*}_1^p)$.
\end{proof}

Figure \ref{fig:vizpthprox} shows the symmetric properties of the proximal operator, and we summarize Lemma~\ref{lemma:prox} in the  Algorithm \ref{alg:pthprox} which shows how to compute $\prox_{\lambda||\cdot||^p_1}$ for $p=1,2,3,4$.

\begin{figure}[ht]
    \centering    \includegraphics[width = .6\textwidth]{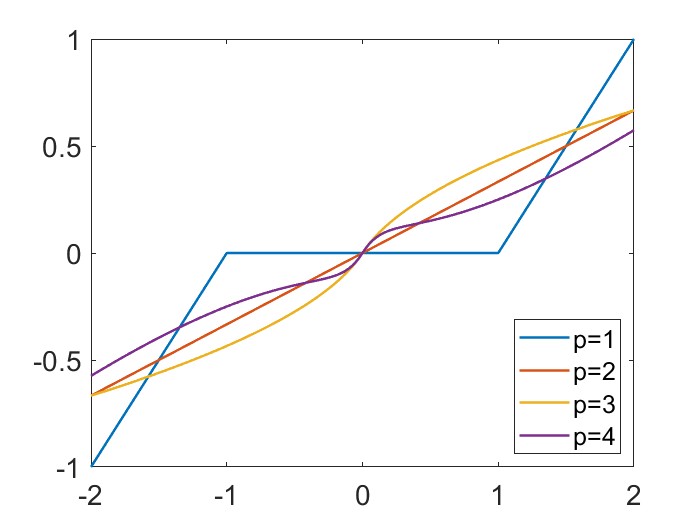}
    \caption{Visualization of the shrinkage of $\prox_{\lambda||\cdot||^p_1}(x)$ for $x\in \R$ and $p=1,2,3,4$.}
    \label{fig:vizpthprox}
\end{figure}

\begin{algorithm}[ht]
\caption{Computing $\prox_{\lambda||\cdot||^p_1}(\Z)$ for a tensor $\Z$}\label{alg:pthprox}
\begin{algorithmic}
\State\textbf{Input:} $\Z\in\R^{n_1\times n_2\times ...\times n_m}$, $p\in\{1,2,3,4\}$, and $\lambda$.
\State\textbf{Output:} $\prox_{\lambda||\cdot||^p_1}(\Z)$
\If{$p=1$}
\State $\prox_{\lambda \norm{\cdot}_1}(\Z)= \sign(\Z)\odot\max\left\{|\Z| - \lambda ,\mathcal{O}\right\}$
\EndIf
\If{$p=2$}
\State $\prox_{\lambda \norm{\cdot}_1^2}(\Z) = \sign(\Z)\odot\max\left\{\left(|\Z| - {2}\lambda \frac{\|\Z\|_1}{1+2\lambda n} \right),\mathcal{O}\right\}$
\EndIf
\If{$p=3$}
\State $\prox_{\lambda \norm{\cdot}_1^3}(\Z) = \sign(\Z)\odot\max\left\{\left( |\Z| - {3}\lambda \left(\frac{-1 + \sqrt{1+12n\lambda \|\Z\|_1}}{6n\lambda}\right)^2\right),\mathcal{O}\right\}$
\EndIf
\If{$p=4$}
\State calculate $P(\Z) = 9 n^2 \lambda^2 \|\Z\|_1 + \sqrt{3} \sqrt{n^3 \lambda^3 (27 n \lambda \|\Z\|_1^2 + 4)}$
\State $\prox_{\lambda \norm{\cdot}_1^4}(\Z) = \sign(\Z)\odot\max\left\{\left( |\Z| - 4\lambda \left(\frac{
    3^{\frac{1}{3}} \left(-3^{\frac{1}{3}} n \lambda + P(\Z)^{\frac{2}{3}}\right)
}{
    6 n \lambda \sqrt[3]{P(\Z)}
}\right)^3\right),\mathcal{O}\right\}$
\EndIf
\end{algorithmic}
\end{algorithm}
\subsection{Sparse Tensor Recovery} We now focus on a sparsity-promoting regularization using the $p$-th power of the $\ell_1$-norm. Given $\A\in\R^{n_1\times n_2\times n_3\times ...\times n_m}$ and $\B\in\R^{n_1\times l\times n_3\times ...\times n_m}$, our goal is to recover a sparse tensor $\X\in\R^{n_2\times l\times n_3\times ...\times n_m}$ such that the linear relationship  $\A*_L\X=\B$ is satisfied. Now we propose a constrained minimization problem based on the power of $\ell_1$-norm regularization:
\begin{align}\label{pth}
\min_{\X}\lambda\norm{\X}^p_{1}+\frac12\norm{\X}_F^2\quad\mbox{s.t.}\quad
\A*_{L}\X=\B,
\end{align}
where $\lambda>0$ is a regularization parameter.

The linear constraint $\A*_L\X = \B$ can be split into $n_1$ constraints by Corollary \ref{cor:refsep},
\begin{align}
    \A(i)*_L\X = \B(i),\quad i\in[n_1].
\end{align}
We define the row space as, $R(\A):=\{\A^T*_L\Y:\Y\in\R^{n_1\times l\times n_3\times ...\times n_m}\}$
and then we assume the optimal solution $\hat{\X}$ to the problem \eqref{pth} satisfies
\[\A*_L\hat{\X}=\B, \hspace{10pt } \partial f(\hat{\X})\cap R(\A)\neq \emptyset.
\]
Given that the objective function in \eqref{pth} exhibits strong convexity under specific conditions, and considering the linear nature of the constraint, we have devised an effective algorithm inspired by the regularized Kaczmarz algorithm \cite{chen2021regularized}. This approach is particularly advantageous for large datasets. Consequently, this leads to two subproblems, one of which involves the application of the proximal operator of the $\ell_1^p$ regularizer.

\begin{algorithm}[H]
\caption{Randomized $\ell_1^p$ regularized Kaczmarz for sparse tensor recovery (L1PK-S)}\label{alg:pth}
\begin{algorithmic}
\State\textbf{Input:} $\B\in\R^{n_1\times l\times n_3\times ...\times n_m}$, $\A\in\R^{n_1\times n_2\times n_3\times...\times n_m}$, $\lambda$, $p$, stepsize $t$, max number of iterations $T$, invertible linear transform $L$, and tolerance $tol$.
\State\textbf{Output:} An approximation of $\hat{\X}$
\State\textbf{Initialize:} $\Z^{(0)}\in R(\A)\subset \R^{n_2\times l\times n_3\times ...\times n_m}, \X^{(0)} = \prox_{\lambda\norm{\cdot}_1}(\Z^{(0)})$
\For{ $k=0,...,T-1$ }
\State choose index $i(k)$ cyclically or randomly from $[n_1]$
\State $\Z^{(k+1)} = \Z^{(k)}+t\A^*(i(k))*_L\frac{\B(i(k))-(\A*_L\X^{(k)})(i(k))}{||\A(i(k))||_F^2}   $
\State \vspace{2pt}$\X^{(k+1)} = \prox_{\lambda\norm{\cdot}^p_1}(\Z^{(k+1)}) $
\State{Terminate if }$||\X^{(k+1)}-\X^{(k)}||_F/||\X^{(k)}||_F<tol$
\EndFor
\end{algorithmic}
\end{algorithm}

\subsection{Low Rank Tensor Recovery}
The concept of a matrix's low-rank property, which is evidenced by the sparsity of its singular values from SVD, can be extrapolated to tensors. By utilizing the $\ell_1^p$ regularized Kaczmarz algorithm for sparse tensor recovery, we adapt this model for the recovery of low-rank tensors by applying the $\ell_1^p$ regularizer to the middle f-diagonal tensor in the t-SVD form of a tensor.
\begin{definition}
    The Nuclear $\ell_1^p$ norm (N$\ell_1^p$) is defined as
    \[\begin{aligned}
\|\X\|_{\text{N}\ell_1^p} = \norm{\S}^p_{1},
    \end{aligned}\]
when $\X^*= \U*_L\S*_L\V^T$ is a t-SVD representation of $\X$.
\end{definition}

We thus consider the regularized minimization problem
\begin{align}
\argmin_{\X}\lambda\|\X\|_{\text{N}\ell_1^p} +\frac12\norm{\X}_F^2\quad\mbox{s.t.}\quad
\A*_{L}\X=\B,
\label{pthsing}
\end{align}
when $p=1$ this reduces to the classical low-rank tensor recovery problem and is solved using the tensor singular value thresholding operator (t-SVT) \cite{qin2022low}. In this work, we focus on  the cases when $p>1$. An explicit form of the proximal operator for the N$\ell_1^p$ norm can be derived in a similar manner to \cite[Theorem 1]{henneberger2023log}.

\begin{definition}
    The tensor singular value proximal operator of the N$\ell_1^p$ norm at $\Z\in \R^{n_1\times n_2\times...\times n_m}$ is defined as
    \begin{align}
\prox_{\lambda\text{N$\ell_1^p$ }}(\Z)=\argmin_{\X}\lambda  \|\X\|_{\text{N$\ell_1^p$ }}+\frac12\norm{\X-\Z}_F^2.
\label{pthsinggen}
\end{align}
\end{definition}
\begin{theorem}\label{thm:lspsingprox}
For each $\Z\in \R^{n_1\times n_2\times...\times n_m}$, if $\X^*$ is a solution to \eqref{pthsinggen}, then there exists a pair $(\U,\V)\in\mathcal{P}(\mathcal{Z})$ such that $\Z = \U*_L \S *_L \V^*$ and a tensor $\D\in \prox_{\lambda\|\cdot\|_1^p }(\S)$ such that $\X^*= \U*_L\D*_L\V^T$.
\end{theorem}
\begin{remark} This is an extension of the proof of Theorem $4$ in \cite{prater2022proximity}. We show that $\X^*= \U*_L\D*_L\V^T$ is a solution to \eqref{pthsinggen}. Note that \eqref{pthsinggen} can be reformulated as
 \begin{align}
\min_{\D}\left(\min_{\X, \X = \U*_L\D*_L\V^T}  \lambda \|\D\|_1^p+\frac12\norm{\X-\Z}_F^2\right).
\label{LSPsinggenref}
\end{align}
where $\D$ is the core tensor in the t-SVD representation of $\X$ . We can verify that
\begin{align}
    \|\X-\Z\|_F^2  =
     &\|\X\|_F^2-2\langle\X,\Z\rangle+\|\Z\|_F^2\nonumber\\
      \geq&\|\D\|_F^2-2\langle \D,\S\rangle+\|\S\|_F^2 = \|\D-\S\|_F^2\label{normsimp},
\end{align}
where the last inequality is due to the von Neumann's trace inequality and $\S$ is the core tensor in the t-SVD representation of $\Z$.
Then the optimization problem \eqref{LSPsinggenref} reduces to
\begin{align}
\min_{\D}  \lambda \|\D\|_1^p+\frac12\norm{\D-\S}_F^2.
\label{pthsingsep}
\end{align}
The solution to this objective function is now $\prox_{\lambda \|\cdot\|_1^p}(\S)$ from above.
\end{remark} Similar to the sparse case above, we adapt the regularized Kaczmarz algorithm from \cite{chen2021regularized} to the current $*_L$-based tensor product and consider the regularization term as the $\ell_1^p$ norm applied to the core tensor from t-SVD. Thus we propose a new algorithm for solving the low-rank tensor recovery problem \eqref{pthsing}, which is described in Algorithm~\ref{alg:pthsing}.

\begin{algorithm}[H]
\caption{Randomized $\ell_1^p$ regularized Kaczmarz for low-rank tensor recovery (L1PK-L)}\label{alg:pthsing}
\begin{algorithmic}
\State\textbf{Input:} $\B\in\R^{n_1\times l\times n_3\times ...\times n_m}$, $\A\in\R^{n_1\times n_2\times n_3\times...\times n_m}$, stepsize $t$, $\lambda$, p, max number of iterations $T$, invertible linear transform $L$, and tolerance $tol$.
\State\textbf{Output:} An approximation of $\hat{\X}$
\State\textbf{Initialize:} $\Z^{(0)}\in R(\A)\subset \R^{n_2\times l\times n_3\times ...\times n_m}, \X^{(0)} = \text{t-SVT}_{\lambda}(\Z^{(0)})$
\For{ $k=0,...,T-1$ }
\State choose index $i(k)$ cyclically or randomly from $[n_1]$
\State $\Z^{(k+1)} = \Z^{(k)}+t\A^*(i(k))*_L\frac{\B(i(k))-(\A*_L\X^{(k)})(i(k))}{||\A(i(k))||_F^2}   $

\State $\X^{(k+1)} =\prox_{\lambda \text{N}\ell_1^p}(\Z^{(k+1)}) $
\State{Terminate if} $||\X^{(k+1)}-\X^{(k)}||_F/||\X^{(k)}||_F<tol$
\EndFor
\end{algorithmic}
\end{algorithm}

\section{Convergence Analysis}\label{sec:Converg}
In this section, we detail the convergence guarantees for the proposed methods. By following the  convergence analysis in \cite{chen2021regularized,henneberger2023log}, we describe our main convergence guarantees for the proposed algorithms in Theorem \ref{thm:cycconv} and Theorem \ref{thm:randconv}, corresponding to Algorithm \ref{alg:pth} when the index $i(k)$ is selected cyclically and randomly, respectively. Detailed proofs can be found in the Appendix.

To start with, one can see that the tensor function of the form
    \begin{align}
        f(\X) =\lambda\norm{\X}^p_{1}+\frac12\norm{\X}_F^2 \label{eqn:feqn}
    \end{align}is $1$-strongly convex since $\frac12\norm{\X}_F^2$ is $1$-strongly convex and $\lambda\norm{\X}^p_{1}$ is convex for $p\in\mathbb{N}$ \cite{boyd2004convex}.
By using the result in \cite[Proposition 1]{henneberger2023log} and  following \cite{chen2021regularized}, we can obtain the following convergence guarantees for Algorithm~\ref{alg:pth}.

\begin{theorem}\label{thm:cycconv}
  Let $f$ be defined by \eqref{eqn:feqn}. If $t<2/\rho$, then the sequence $\{\X^{(k)}\}$ generated by Algorithm~\ref{alg:pth} with a cyclical control sequence satisfies
\begin{align}
    f(\X^{(k)})-f(\X^{(k+1)})&\leq \langle \Z^{(k+1)},\X-\X^{(k+1)}\rangle -\langle \Z^{(k)},\X-\X^{(k)}\rangle \ \\
&- t\left(1-\frac{t\rho}{2}\right)\frac{\| A(i(k))*_L(\X^{(k)}-\X)\|_F^2}{\|\A(i(k))\|_F^2}
\end{align}
Moreover, the sequence $\{\X^{(k)}\}$ converges to the solution of \eqref{pth}.
\end{theorem}

Likewise, when the control sequence is chosen randomly, a linear convergence rate in expectation can be shown by adapting the proof \cite{chen2021regularized}. We assume that the probability of choosing index $i(k)=j$ from $[n_1]$ in Algorithm \eqref{alg:pth} is proportional to $\|\A(j)\|_F^2$ which is a customary choice in Kaczmarz literature \cite{strohmer2009randomized}. Thus for Algorithm \ref{alg:pth}, we have the following result.

\begin{theorem}\label{thm:randconv} Consider $f$ as defined in \eqref{eqn:feqn}. If $0<\frac{\nu t}{\|\A\|_F^2}(1-\frac{t \rho}{2})<1$, then there exists $\nu>0$ such that the sequence $\{\X^{(k)}\}$ generated by Algorithm~\ref{alg:pth} with random $i(k)$'s converges linearly to the solution $\hat{\X}$ in expectation,
\[\mathbb{E}\|\X^{(k)}-\hat{\X}\|_F^2\leq \|\hat{\X}-\X^{(0)}\|_F^2\left(1-\frac{\nu t}{\|\A\|_F^2}(1-\frac{t \rho}{2})\right)^k.\]
\end{theorem}

Note that all the aforementioned convergence results can be extended for Algorithm~\ref{alg:pthsing}.

\section{Extensions}\label{sec:exten}

\subsection{Block Kaczmarz} To enhance the performance of our algorithms for handling large-scale systems of equations, we introduce a modified version in this section, namely, the block variant. The block Kaczmarz method partitions the row indices of $\A$ into a predetermined number of distinct blocks. In the context of matrix-based Kaczmarz algorithms, employing block strategies that utilize distinct, non-overlapping indices has been shown to boost the efficiency of the traditional RKA method's convergence rates \cite{needell2014paved,needell2015randomized,necoara2019faster}. Using Theorem \ref{thm:cycconv}, a linear convergence rate in expectation can be shown for Algorithm \ref{alg:block}.

\begin{theorem}\label{thm:blockconv}
Consider the problem \eqref{pth}. If $0<\frac{\nu t}{M\|\A\|_F^2}(1-\frac{t \rho}{2})<1$, then $\X^{(k)}$  converges linearly to the solution $\hat{\X}$ in expectation:
    \begin{align*}
    \mathbb{E}\|\X^{(k)}-\hat{\X}\|_F^2
    &\leq  \|\hat{\X}-\X^{(0)}\|_F^2\left(1-\frac{\nu t}{M\|\A\|_F^2}(1-\frac{t \rho}{2})\right)^k.
\end{align*}
Note that this result relies on $M$, the number of blocks.
\end{theorem}
\begin{remark}
    The proof of Theorem \ref{thm:blockconv} follows as in \cite[Theorem 4]{henneberger2023log}.
\end{remark}

\begin{algorithm}
\caption{Block randomized regularized Kaczmarz for tensor recovery with $L$-based t-product}\label{alg:block}
\begin{algorithmic}
\State\textbf{Input:} $\B\in\R^{n_1\times l\times n_3\times ...\times n_m}$, $\A\in\R^{n_1\times n_2\times n_3\times...\times n_m}$, stepsize $t$, max number of iterations $T$, partition $P = \{\tau_1,...,\tau_M\}$ of $[n_1]$, invertible linear transform $L$, and tolerance $tol$.
\State\textbf{Output:} An approximation of $\hat{\X}$
\State\textbf{Initialize:} $\Z^{(0)}\in R(\A)\subset \R^{n_2\times l\times n_3\times ...\times n_m}, \X^{(0)} = \nabla f^*(\Z^{(0)})$
\For{$k=0,...,T-1$}
\State choose a block $\tau_{k+1}$ cyclically or randomly from $P$
\State $\Z^{(k+1)} = \Z^{(k)}+t\A^*(\tau_k)*_L\frac{\B(\tau_k)-(\A*_L\X^{(k)})(\tau_k)}{||\A(\tau_k)||_F^2}   $
\State $\X^{(k+1)} = \nabla f^*(\Z^{(k+1)}) $
\State {Terminate if }$||\X^{(k+1)}-\X^{(k)}||_F/||\X^{(k)}||_F<tol$
\EndFor
\end{algorithmic}
\end{algorithm}

\subsection{Accelerated Kaczmarz}
In the Kaczmarz algorithm, the step size $t$ can be either a predefined constant or determined adaptively during iterations. Adaptive methods for setting $t$ include line search, trust region methods, or the Barzilai-Borwein method \cite{barzilai1988two}. Furthermore, the Nesterov acceleration technique \cite{nesterov1983method} has been effectively incorporated into gradient descent algorithms. This technique, known as Nesterov Accelerated Gradient Descent (NAGD), is notable for its enhanced convergence rate of $O(1/N^2)$ compared to the typical $O(1/N)$ rate of gradient descent. To integrate this acceleration technique into our algorithm at the k-th iteration, the step size $t$ is updated according to Algorithm \ref{alg:accel}, starting with an initial $\gamma_0=1$ and update $\Z$ as a linear combination of the current and previous steps, weighted by $\eta_k$.

\begin{algorithm}
\caption{Accelerated block randomized regularized Kaczmarz for tensor recovery with new t-product}\label{alg:accel}
\begin{algorithmic}
\State\textbf{Input:} $\B\in\R^{n_1\times l\times n_3\times ...\times n_m}$, $\A\in\R^{n_1\times n_2\times n_3\times...\times n_m}$, step size $t$, max number of iterations $T$, partition $P = \{\tau_1,...,\tau_m\}$ of $[n_1]$, invertible linear transform $L$, and tolerance $tol$.
\State\textbf{Output:} An approximation of $\hat{\X}$
\State\textbf{Initialize:} $\Z^{(0)}\in R(\A)\subset \R^{n_2\times l\times n_3\times ...\times n_m}, \X^{(0)} = \nabla f^*(\Z^{(0)}), \gamma_0=1$
\For{$k=1,...,T-1$}
\State choose a block $\tau_k$ cyclically or randomly from $P$
\State $\gamma_{k+1} = \frac{1+\sqrt{1+4\gamma_k^2}}{2}$
\State $\eta_{k+1} = \frac{1-\gamma_k}{\gamma_{k+1}}$
\State $\hat{\Z}^{(k+1)} = \Z^{(k)}+t\A^*(\tau_k)*_L\frac{\B(\tau_k)-(\A*_L\X^{(k)})(\tau_k)}{||\A(\tau_k)||_F^2}   $
\State $\Z^{(k+1)} = (1-\eta_{k+1})\hat{\Z}^{(k+1)} + \eta_{k+1}\hat{\Z}^{(k)}$
\State $\X^{(k+1)} = \nabla f^*(\Z^{(k+1)}) $
\State {Terminate if }$||\X^{(k+1)}-\X^{(k)}||_F/||\X^{(k)}||_F<tol$
\EndFor
\end{algorithmic}
\end{algorithm}

\section{Numerical Experiments}\label{sec:num}

In this section, we illustrate the success of our proposed algorithms on synthetic and real-world data for sparse and low-rank high-order tensor recovery applications. Our numerical experiments utilize several quantitative metrics to assess the performance of different tensor recovery algorithms:
\begin{enumerate}
    \item The relative error (RE) at each iteration is defined as
 \begin{align}
    \mbox{RE}(\X^{(k)},\hat{\X})= \frac{\|\hat{\X}-\X^{(k)}\|_F}{\|\hat{\X}\|_F},\label{relative error}
 \end{align}
 where $\X^{(k)}$ represents the estimation of the ground truth tensor $\hat{\X}$ at iteration $k$.
 \item The Peak Signal-to-Noise Ratio (PSNR)\cite{qin2022low}  is computed as
\[\text{PSNR} = 10 \log_{10}(n_1\times n_2\times \cdots \times n_m\|\X\|_\infty^2/\|\hat{\X}-\X\|_F^2),\]
where $\X$ is the original tensor of size $n_1\times n_2\times ...\times n_m$ and $\hat{\X}$ is the approximated tensor. This formula extends the conventional PSNR used in image processing.
\item The Structural Similarity Index Measure (SSIM) \cite{wang2004image} calculates feature similarity by combining similarity measures for luminance, contrast, and structure.
\item The Feature Similarity Index Measure (FSIM) \cite{zhang2011fsim} uses phase congruency and gradient magnitude to characterize local image quality.
\end{enumerate}

For the synthetic experiments, we use the relative error \eqref{relative error} for performance evaluation. However, for real-world data experiments, we incorporate all four metrics. All experiments were conducted using MATLAB 2023b on a desktop computer with Intel CPU i7-1065G7 CPU, 12GB RAM, running Windows 11. The demo codes for our algorithms are available at \hyperlink{https://github.com/khenneberger}{https://github.com/khenneberger}.

\subsection{Synthetic Numerical Experiments}
 In the following numerical experiments we apply the regularized Kaczmarz algorithm with blocking. The number of blocks, $M$, is preassigned and the row indices $[n_1]$ are partitioned into $M$ blocks at the start of the algorithm. In this case, a different index block $i(k)$ is chosen at each iteration $k$. Our choice of indexing can be cyclic or randomized. Since the number of blocks effects convergence, we may tune the block number to find the optimal parameter.

For the following synthetic experiments we consider the general minimization problem \eqref{pth} where the tensors $\A\in\R^{50\times 2\times 50\times 50}$ and $\X\in\R^{2\times 50\times 50\times 50}$ are randomly generated. Throughout this section we set the parameters as $M= 3, \lambda = 0.001$, $t=1$, and $L=$ \verb 'fft' unless otherwise specified.

\begin{figure}
\setlength{\tabcolsep}{2pt}
    \centering
     \begin{tabular}{c c}
     \includegraphics[width = .49\textwidth]{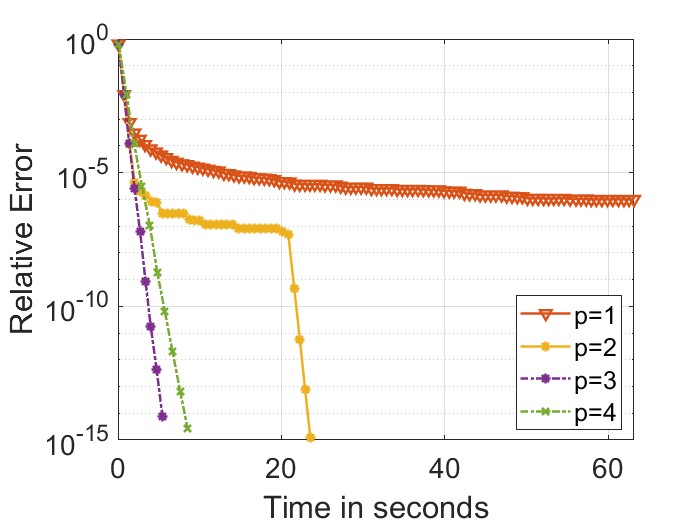}  & \includegraphics[width = .49\textwidth]{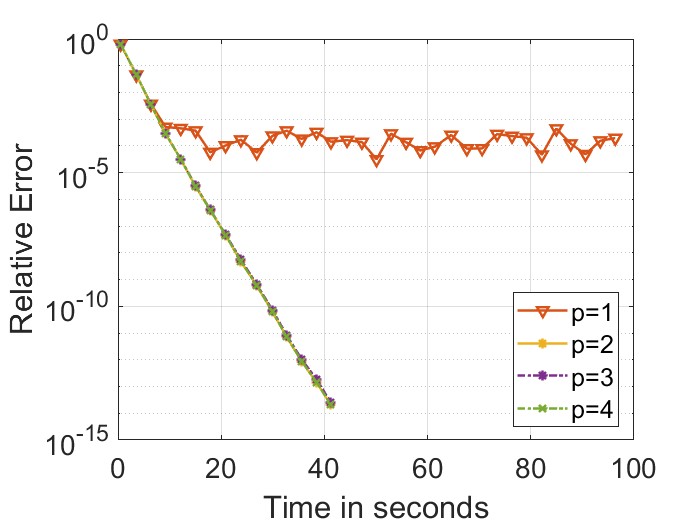} \\
        L1PK-S (Alg.~\ref{alg:pth}) &  L1PK-L (Alg.~\ref{alg:pthsing})
     \end{tabular}
     \caption{Convergence for different $p$ values. The sparse data is created with a sampling rate of $80\%$}
     \label{fig:pthcomp}
 \end{figure}
First, we compare the convergence of our two proposed Algorithms \ref{alg:pth} and \ref{alg:pthsing} for different powers of $p$. Figure \ref{fig:pthcomp} plots the relative error curves produced by using the $p=\{1,2,3,4\}$. Next, we compare the convergence of the proposed Algorithms \ref{alg:pth} and \ref{alg:pthsing}  with different linear transforms: the Fast Fourier Transform (FFT), the Discrete Cosine Transform (DCT), and the Discrete Wavelet Transform (DWT) with the Daubechies 5 wavelet `\verb"db5"'.  Figure \ref{fig:lintrans} plots the relative error curves produced by using the linear transforms FFT, DCT, and DWT. For this synthetic data with the aforementioned parameters, we fix the $\rho$ values as 1 for the first plot and as $n_3\cdot n_4 = 2500$ for the second plot. When $\rho=1$ this method is slower to converge than when $\rho = n_3\cdot n_4 = 2500$.
 \begin{figure}
 \setlength{\tabcolsep}{2pt}
     \centering
     \begin{tabular}{c c}
        \includegraphics[width = .49\textwidth]{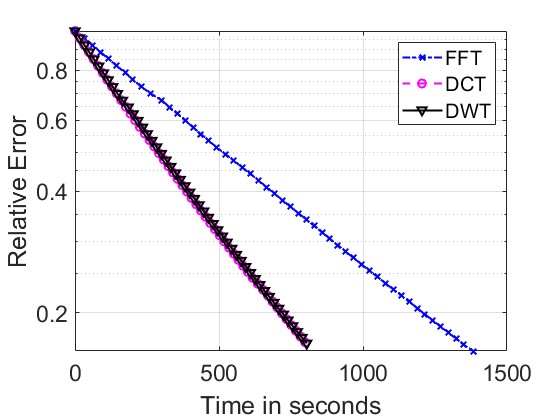}  &  \includegraphics[width = .49\textwidth]{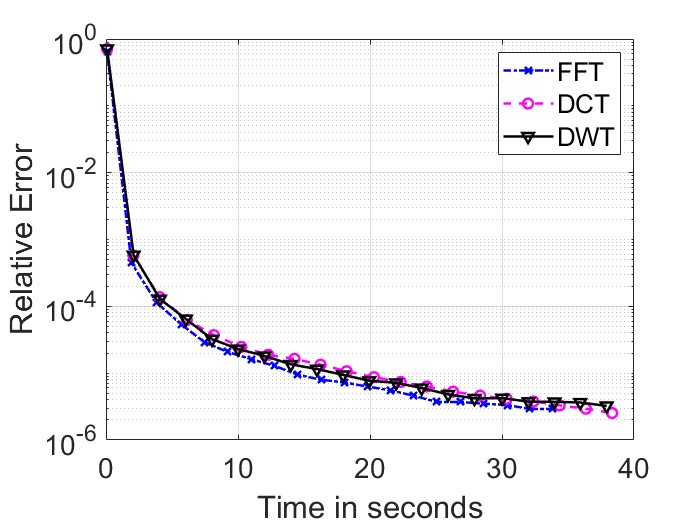} \\
       L1PK-S (Alg.~ \ref{alg:pth}) with $\rho = 1$ &  L1PK-S (Alg.~ \ref{alg:pth}) with $\rho = n_3\cdot n_4$
     \end{tabular}
     \caption{Convergence under different linear transforms}
     \label{fig:lintrans}
 \end{figure}

We also compare the convergence of our algorithms with different block sizes. Figure \ref{fig:blocksize} plots the relative error curves versus the time in seconds for our two algorithms where blocking is performed cyclically.

\begin{figure}[H]
 \setlength{\tabcolsep}{2pt}
     \centering
     \begin{tabular}{c c}
        \includegraphics[width = .49\textwidth]{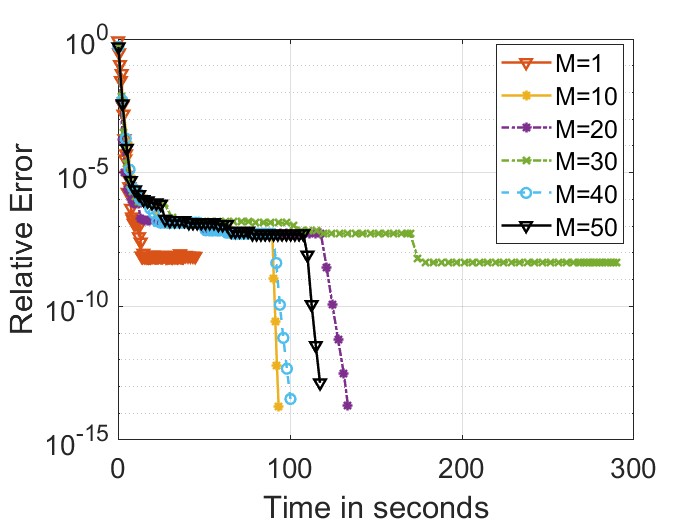}
        &  \includegraphics[width = .49\textwidth]{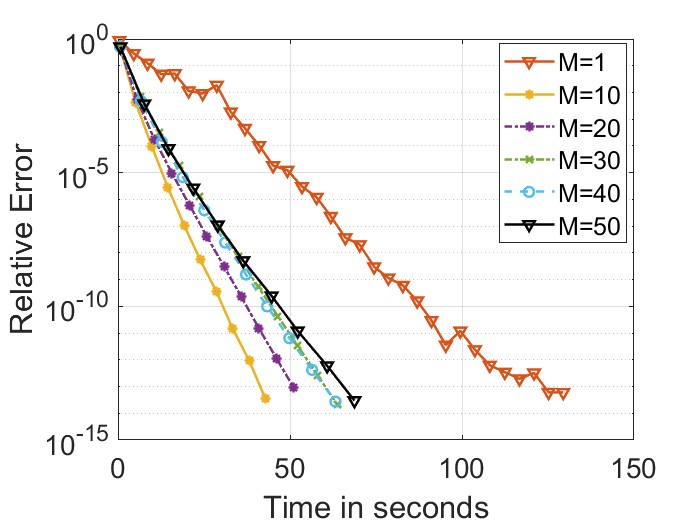} \\
       L1PK-S (Alg.~ \ref{alg:pth}) &  L1PK-L (Alg.~ \ref{alg:pthsing})
     \end{tabular}
     \caption{Convergence with different block size $M$. }
     \label{fig:blocksize}
 \end{figure}

Our proposed block algorithm \ref{alg:pth} consistently yields stable results across various conditions. However, the accelerated algorithm \ref{alg:accel} demonstrates superior performance under specific circumstances. Notably, when the lambda value is small and the block size is large, the accelerated algorithm proves to be more advantageous. This benefit is depicted in Figure \ref{fig:accel}, which presents a comparison of relative error curves against elapsed time for both algorithms.

Additionally, an alternative version of Nesterov acceleration was examined. This variation involved updating the momentum as $k/(k+3)$ at each iteration $k$. Despite its novelty, this second version of the accelerated algorithm did not generally surpass the other methods. It only showed improved performance over other versions with the parameters $p=2$ and block size $M=2$, and when $\lambda$ is small.

\begin{figure}
     \centering
        \includegraphics[width = .49\textwidth]{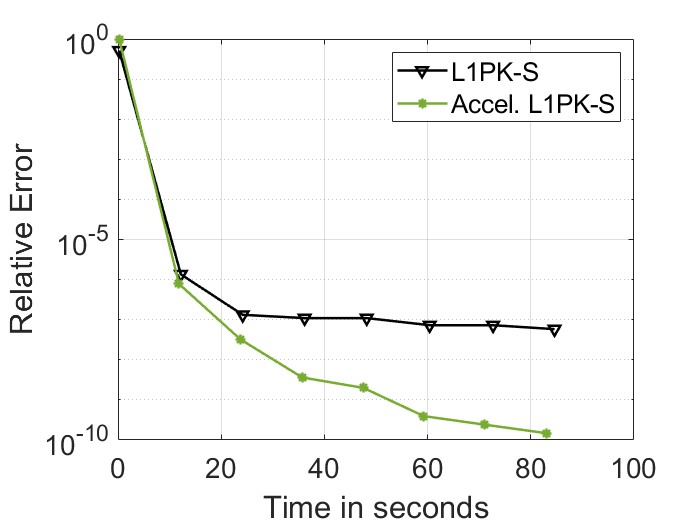}
     \caption{Accelerated versus regular convergence with parameters: $t = 1, \lambda = 0.001, p= 2$, block size $M= 25$.}
     \label{fig:accel}
 \end{figure}

\subsection{Real data experiments}
In this subsection, we illustrate the performance of the proposed
algorithms in two real-world applications: image destriping and color video deblurring. We analyze the performance of our algorithms in comparison to other state-of-the-art methods using the relative error, PSNR, SSIM, and FSIM metrics previously defined.
\subsubsection{Image sequence destriping}
In this experiment we consider a low-rank image destriping problem. Striping is a common issue in scientific imaging, thus image destriping aims to remove stripes or streaks from the visual data. We consider Extended YaleFace Dataset B \cite{georghiades2001few}, a 4th order dataset of 38 subjects photographed under different illumination conditions. As in \cite{qin2022low} we consider a subset of the original data such that our tensor is of size $48\times 42\times 64\times 38$. This includes $38$ subjects under $64$ different illuminations where each image is $48\times 42$ pixels.

We choose the tensor $\A$ to be facewise diagonal with all diagonal entries equal to one except in the rows which will establish the striping effect.
Moreover, we set the values of the diagonal entries  where we wish to create a stripe as $0.01$. After parameter tuning, we obtain the optimal parameters $\lambda = 0.001, t = 1, p=1$ and $ \beta = 1$. Figure \ref{fig:destripe} shows the results on three selected images, demonstrating both the quantitative and qualitative success of this algorithm.

Table \ref{tab:destripecomp} details the quantitative comparison of our method with two state of the art tensor completion algorithms: HTNN-FFT \cite{qin2022low} and t-CTV \cite{wang2023guaranteed}. Notably, our proposed method performs better in terms of PSNR and relative error.

\begin{figure}[ht]
    \centering
    \setlength{\tabcolsep}{2pt} 
    \begin{tabular}{ccccc}
        Original & Observed &  t-CTV& HTNN-FFT & Proposed \\
        {\includegraphics[width = 0.17\textwidth]{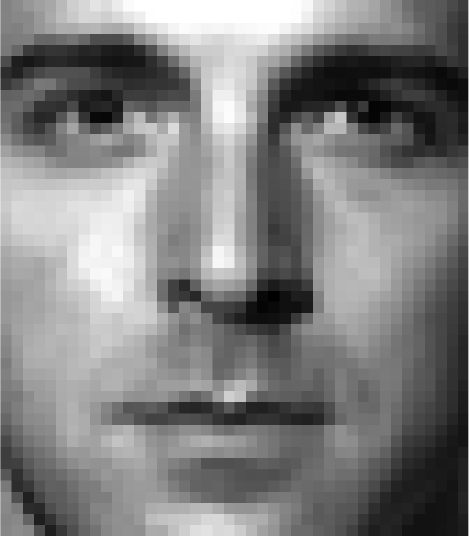}}
        & {\includegraphics[width = 0.17\textwidth]{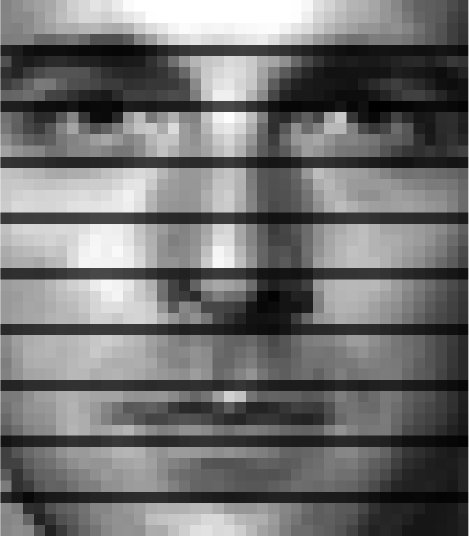}}
        &{\includegraphics[width = 0.17\textwidth]{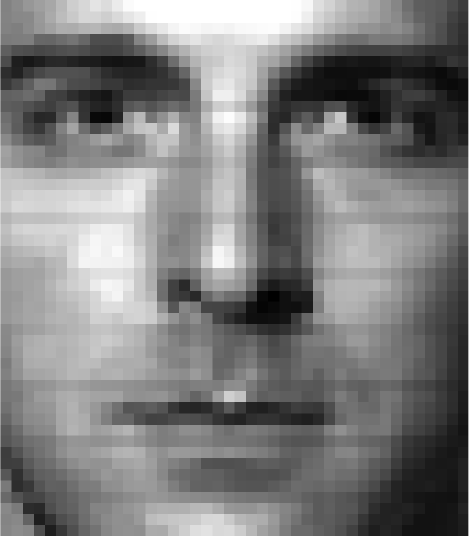}}
        & {\includegraphics[width = 0.17\textwidth]{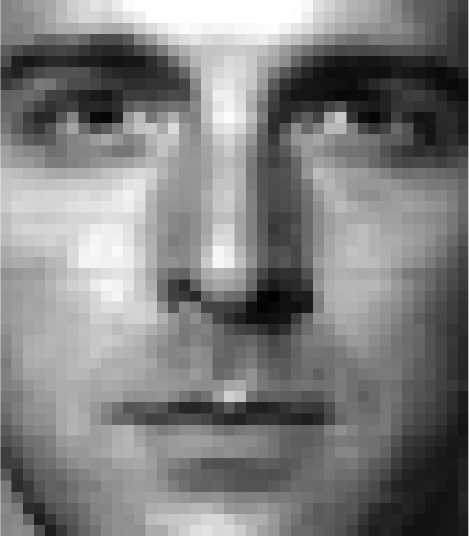}}
        & {\includegraphics[width = 0.17\textwidth]{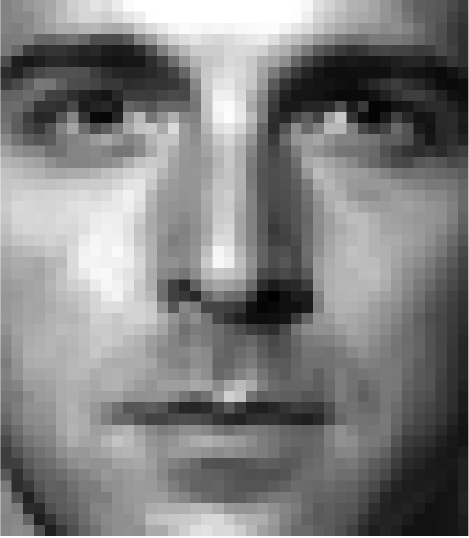}} \\
      {\includegraphics[width = 0.17\textwidth]{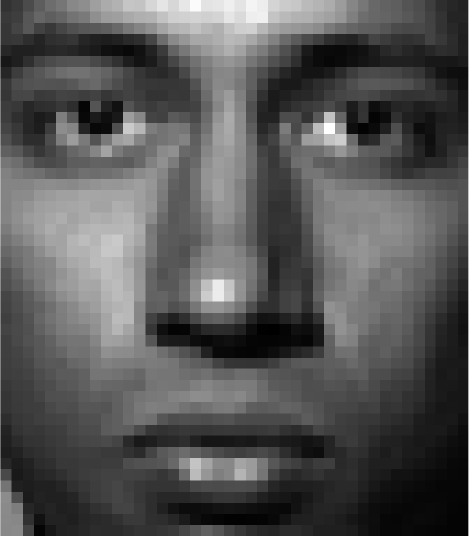}}
      & {\includegraphics[width = 0.17\textwidth]{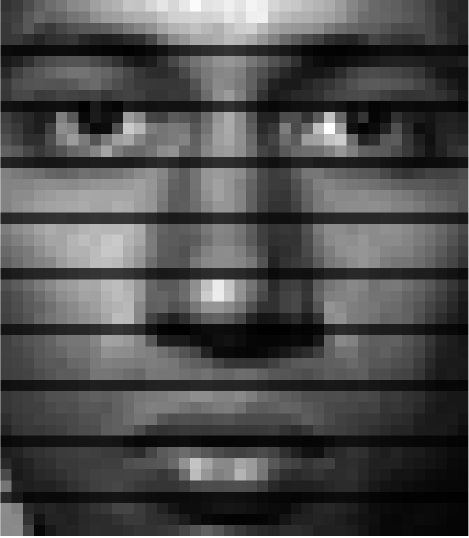}}
      &  {\includegraphics[width = 0.17\textwidth]{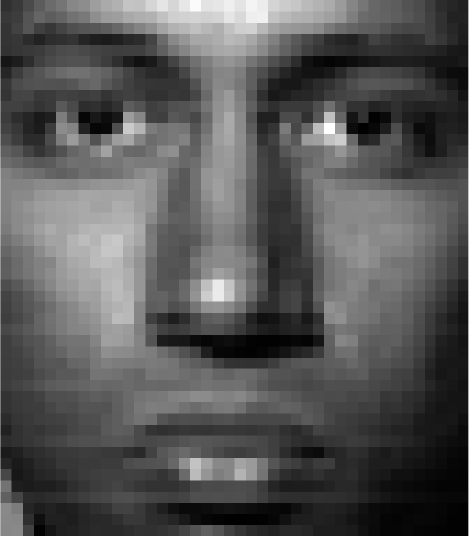}}
      & {\includegraphics[width = 0.17\textwidth]{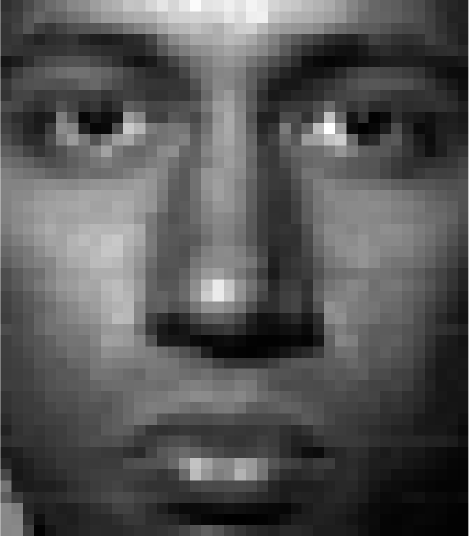}}
      &{\includegraphics[width = 0.17\textwidth]{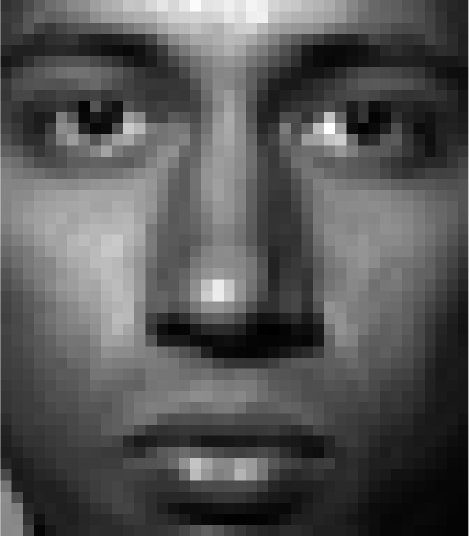}}\\
     {\includegraphics[width = 0.17\textwidth]{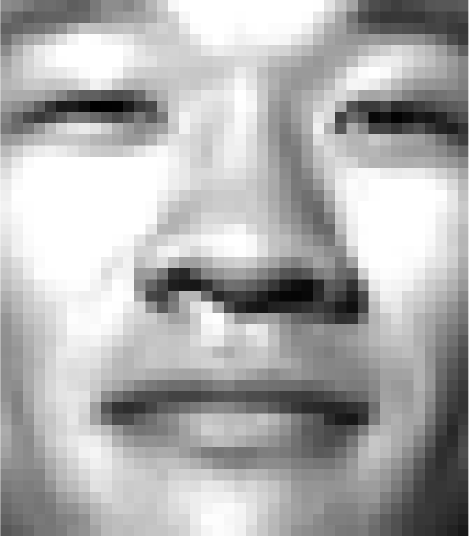}}
     & {\includegraphics[width = 0.17\textwidth]{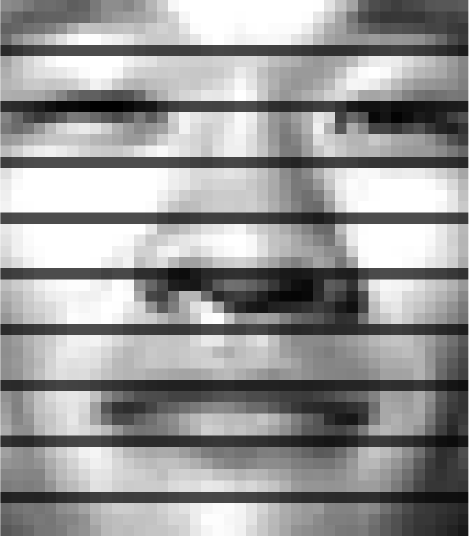}}
     & {\includegraphics[width = 0.17\textwidth]{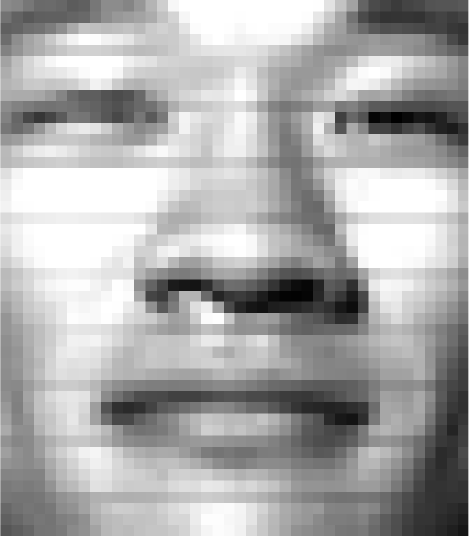}}
     & {\includegraphics[width = 0.17\textwidth]{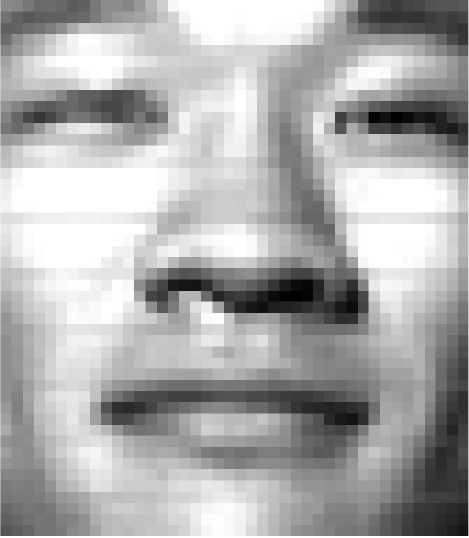}}
     &{\includegraphics[width = 0.17\textwidth]{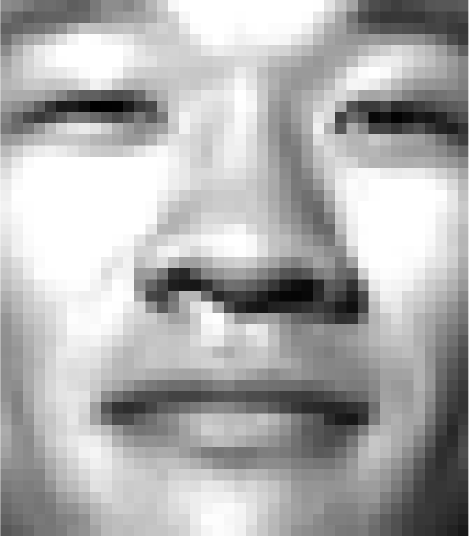}}\\
     & & \scriptsize RE $=0.07591$ & \scriptsize RE $= 0.0542$ & \scriptsize RE $= 8.8652\times 10^{-4}$
        \end{tabular}
        \vspace{-5pt}
        \caption{Visual results where every 5th row is a stripe}
        \label{fig:destripe}
\end{figure}

\begin{table}[H]
    \centering
    \begin{tabular}{|c|c|c|c|c|c|}
\hline
     \textbf{Method} & \textbf{PSNR}  & \textbf{SSIM}  & \textbf{FSIM} & \textbf{RE} & \textbf{time (sec)} \\
     \hline\hline
         t-CTV     &     $31.12$  &  $0.9496$ &   $0.9729$  & $0.07590$  &$266.99$\\
       \hline
      HTNN-FFT  & $34.11$ &$0.9646$  &  $0.9802$   &  $0.05385$& $82.7124$\\
      \hline
       Proposed & $\bf 77.38$   & $\bf 0.9999$  & $\bf{0.9999}$ & $\bf 3.694\times 10^{-4}$ &  $ \textbf{73.3927}$\\
   \hline
\end{tabular}
    \caption{Performance Comparison for $4$th order tensor recovery with row sampling}
    \label{tab:destripecomp}
\end{table}

\subsubsection{4D Video Deconvolution} Consider a third-order tensor $\X\in\R^{n_1\times p\times m_1}$ and a point spread function $H\in \R^{m_2\times n_2}$. Let $m = m_1+m_2-1$ and $n = n_1+n_2-1$. By applying zero padding we may extend $\X$ and $H$ such that they are of size $n_1\times p\times m$ and $m\times n$ respectively. Then we create a doubly block circulant tensor $\A$ of size $n\times n_1\times m$.
Then the relationship between the third-order tensor product and convolution has been established in \cite{chen2021regularized} as
\[
    H\circledast \X = \text{fold}(\text{circ}(\A)\text{unfold}(\X)) = \A*\X. \label{3dconvprod}
\]
Next we establish the equivalence between the t-product and convolution with a 4D tensor. Consider $H$ from before and $\X\in\R^{n_1\times p\times m_1\times q}$. We wish to create $\A$ of size $n\times n_1\times m\times q$ such that $H\circledast \X = \A*\X$. Reducing this to a third order problem, consider $\hat{\X}\in \R^{(n_1 q)\times p\times m}$ obtained by applying ``unfold" to zero-padded $\X$, and consider $\hat{\A}\in \R^{(n q)\times (n_1 q)\times m}$ the doubly block circulant tensor.
Then
\[\begin{aligned}
    \text{fold}(H\circledast \hat{\X}) &\overset{\text{\eqref{3dconvprod}}}{=} \text{fold}( \underbrace{\hat{\A}*\hat{\X}}_{\text{3rd order}}) \overset{\text{ \eqref{convolution}}}{=} \underbrace{\text{circ}^{-1}(\hat{\A})*\X}_{\text{4th order}}.
\end{aligned}\]

Thus the desired $\A =\text{circ}^{-1}(\hat{\A}) $ and is of size $n\times n_1\times m\times q$. To recover $\X$ we consider the low-rank recovery model \eqref{pthsing} where $\B$ of size $ n\times p\times m\times q$ is the tensor version of the observed blurry data.

We now consider a color video recovery problem using a video sequence from the publicly available YUV\footnote{\hyperlink{https://media.xiph.org/video/derf/}{https://media.xiph.org/video/derf/}} database. The benchmark color video is of size $144\times 173\times 3\times 300$. For our experiment we trim the video and reorder the dimensions to size $173\times 3\times 148\times 100$, such that the video consists of $100$ frames, $3$ color channels, and each frame is of size $148\times 173$ pixels with zero padding.  Each blurry image is generated by convolving the ground truth with a Gaussian convolution kernel of size $5\times 5$ with standard deviation $1$. We convolve all frames with the same kernel in its extended tensor form $\A$ of size ${177\times 173\times 148\times 100}$. In this case, the resulting blurry video obtained from the t-product $\A*\X$ is the same as the blurry video obtained by applying the MATLAB function \verb|conv2| to each frame of $\X$.

Figure \ref{fig:blurimgs} presents a comparative analysis of the recovery for the first four frames of the original video using our proposed accelerated method from Algorithm \ref{alg:accel}. This comparison encompasses several methodologies: the Total Variation (TV) image deblurring technique implemented using the Bregman cookbook\footnote{\hyperlink{https://www.mathworks.com/matlabcentral/fileexchange/35462-bregman-cookbook}{https://www.mathworks.com/matlabcentral/fileexchange/35462-bregman-cookbook}} \cite{gilles2011bregman}, the IRCNN method\footnote{\hyperlink{https://github.com/cszn/DPIR}{https://github.com/cszn/DPIR}} \cite{zhang2021plug}, and the Lucy-Richardson (LR) approach \cite{richardson1972bayesian,lucy1974iterative} implemented using the MATLAB function \verb|deconvlucy|. To ensure a fair comparison, the TV image deblurring method includes a preprocessing step where the border of each video frame is padded. We also note that the TV and IRCNN image deblurring methods are inherently three-dimensional, and thus applied to each color video frame. In contrast, the Lucy-Richardson method, being two-dimensional, was executed on each frame and across all color channels for the recovery process.

\begin{figure}[ht]
        \centering
    \setlength{\tabcolsep}{2pt}
       \begin{tabular}{cccccc}
       Original & Observed  & LR & IRCNN & TV & Proposed \\
       {\includegraphics[width = 0.15\textwidth]{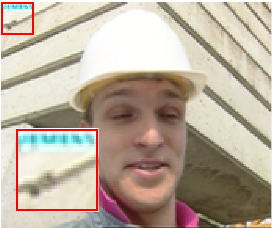}}
       &{\includegraphics[width = 0.15\textwidth]{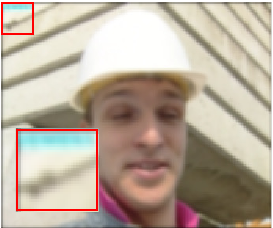}} &
        {\includegraphics[width = 0.15\textwidth]{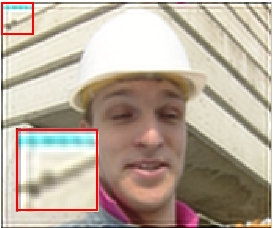}}&
       {\includegraphics[width = 0.15\textwidth]{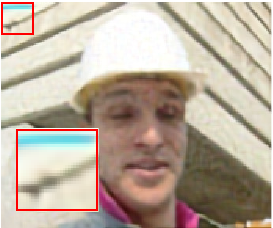}}&
       {\includegraphics[width = 0.15\textwidth]{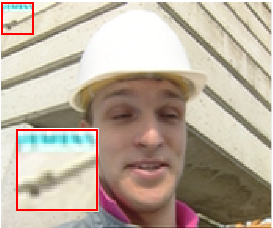}}&
       {\includegraphics[width = 0.15\textwidth]{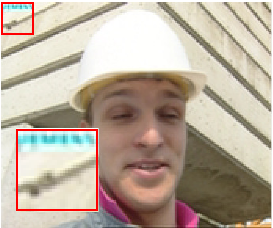}}\\
 {\includegraphics[width = 0.15\textwidth]{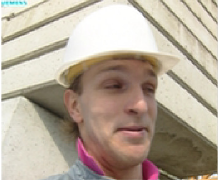}}
 &{\includegraphics[width = 0.15\textwidth]{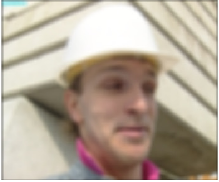}}  &
  {\includegraphics[width = 0.15\textwidth]{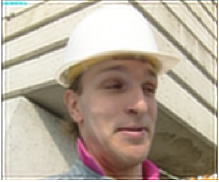}}
  &{\includegraphics[width = 0.15\textwidth]{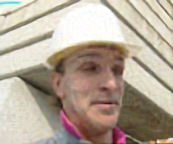}}&
       {\includegraphics[width = 0.15\textwidth]{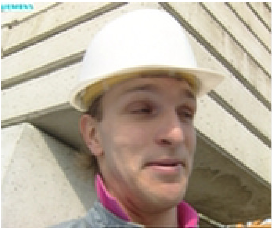}}
       &{\includegraphics[width = 0.15\textwidth]{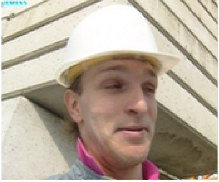}}
        \end{tabular}
        \vspace{-5pt}
        \caption{ Color video deblurring}
        \label{fig:blurimgs}
\end{figure}

\begin{table}[H]
    \centering
    \begin{tabular}{|c|c|c|c|c|}
\hline
     \textbf{Method} & \textbf{PSNR}  & \textbf{SSIM}  & \textbf{FSIM} & \textbf{RE} \\
     \hline\hline
      LR & $26.04$   & $0.9578$  & $0.9184$ & $ 0.0987$ \\
     \hline
      IRCNN & $29.87$   & $0.8804$  & $0.9065$ & $ 0.0458$ \\
     \hline
     TV & $45.67$   & $0.9968$  & $0.9984$ & $ 0.0074$ \\
      \hline
      Proposed  & $\boldsymbol{50.64 }$  & $\boldsymbol{0.9984}$  & $\boldsymbol{0.9996}$ & $\boldsymbol{0.0043}$ \\
      \hline
\end{tabular}
    \caption{Performance comparison for color video deblurring}
    \label{tab:deblurcomp}
\end{table}

\subsection{Discussion}
\paragraph{Parameter Selection} In the analysis of both synthetic and real-world datasets detailed previously, parameter optimization was undertaken through a systematic grid search approach, for the parameters $t$, $\lambda$, $M$, and $p$ such that the assumptions of Theorems \ref{thm:cycconv} and \ref{thm:randconv} hold. As demonstrated in the synthetic experiments, $p=2$ outperforms other $p$ values. It was observed that for the real-world datasets evaluated, a relatively small $\lambda$ value, specifically $\lambda=0.001$, yielded optimal results. It is important to highlight the observed trade-off between computational time and convergence rate in the context of block size selection. This trade-off emphasizes the importance of judicious block size selection to effectively manage the balance between computational demand and the rate of applying the model efficiently. Moreover, using the FFT as the linear transformation $L$ resulted in the best convergence outcomes. Generally, irrespective of the choice of other parameters, a block size of 1 suggests an optimal step size of $t=1$. For block sizes larger than 1, a minor adjustment of the step size might be required for optimal convergence, but we recommend starting with $t=1$.

\paragraph{Computational Complexity} For Algorithm \ref{alg:pthprox}, the overall computational complexity for an input tensor $\Z\in\R^{n_1\times\ldots\times n_m}$ is $O(n_1 n_2 \cdots n_m)$. For Algorithm \ref{alg:pth}, the complexity in each iteration is comprised of two main components: updating $\Z^{(k+1)}$ and updating $\X^{(k+1)}$. The update for $\Z^{(k+1)}$ is primarily determined by the high-order t-product operation. When FFT is selected as the transform $L$, the complexity is $O(n_2 \cdots n_m \log(n_3 \cdots n_m) + n_2 l n_3 \cdots n_m)$. Since $\X^{(k)}$ is of size $n_2\times l\times n_3\times\ldots\times n_m$, we can deduce from the complexity of the proximity operator that updating $\X^{(k+1)}$ will have a complexity of $O(n_2 l n_3 \cdots n_m)$. Similarly, the complexity of Algorithm \ref{alg:pthsing} in each iteration also consists of two parts. Updating $\Z^{(k+1)}$ has the same computational cost as that in Algorithm \ref{alg:pth}. Updating $\X^{(k+1)}$ requires $O(n_2 \cdots n_m \log(n_3 \cdots n_m) + n_2 l n_3 \cdots n_m + n_2 l^2 n_3 \ldots n_m)$. Note that for a general linear transform $L$ with transform matrices $\{U_{n_i}\}_{i=3}^m$, the $*_L$-based t-product introduces a higher computational complexity of $O(n_2(n_3 \cdots n_m)^2 + n_2 l n_3 \cdots n_m)$.

\section{Conclusions and Future Works}\label{sec:Con}
In this paper, we introduce novel regularized models designed to reconstruct tensors characterized by sparse or low-rank structures. These models leverage a generalized t-product for higher-order tensors, based on a general invertible linear transform. To solve these models, we develop regularized Kaczmarz algorithms that alternate between the Kaczmarz step and the proximity operator. We also discuss the convergence guarantees and the computational complexity of our algorithms. Furthermore, we enhance these algorithms by introducing a block variant and an accelerated version, aimed at improving computational efficiency for processing large datasets. Through extensive numerical experiments on both synthetic and real datasets, we demonstrate the superior performance of our algorithms compared to other state-of-the-art methods. In future work we will explore the use of non-integer powers $p$ for our algorithms.

\section*{Acknowledgements}
The research of Qin is supported by the NSF grant DMS-1941197. The research of Henneberger is supported partially by the NSF grant DMS-1941197 and DMS-2110731. In addition, the authors would like to thank Dr. Ding Lu from the University of Kentucky for the helpful discussion.

\appendix

\section{Proof of Theorem \ref{thm:cycconv}}\label{app:a}
\renewcommand{\thesection}{\Alph{section}}
\begin{lemma}
     The t-product as defined in Table ~\ref{tab:definitions} yields two important properties:
\begin{align}
    \langle \A,\B\rangle  &= \frac{1}{\rho}\langle \text{bdiag}(\A_L),\text{bdiag}(\B_L)\rangle \label{property1},\\
      \|\A\|^2_F&=\frac{1}{\rho}\|\text{bdiag}(\A_L)\|_F^2.\label{froprop}
\end{align}
 \end{lemma}
\begin{lemma}
If $\A\in\R^{n_1\times n_2 \times n_3\times ...\times n_m}$ and $\X\in\R^{n_2\times l\times n_3\times ...\times n_m}$, then we have
 \begin{align}
     \|\A*_L\X\|_F\leq \sqrt{\rho}\|\A\|_F\|\X\|_F.\label{convlemma1}
 \end{align}
 \end{lemma}
 \begin{proof}
 By following the definition of the Frobenius norm for tensors, we have
\begin{align*}
    \|\A*_L\X\|_F^2 &= \|L^{-1}(\A_L\Delta\X_L)\|_F^2\\
    & = \frac{1}{\rho}\|\text{bdiag}(\A_L\Delta\X_L)\|_F^2 &\text{by \eqref{property1}}\\
    & = \frac{1}{\rho}\|\text{bdiag}(\A_L)\text{bdiag}(\X_L)\|_F^2\\
    & \leq \frac{1}{\rho}\|\text{bdiag}(\A_L)\|_F^2\|\text{bdiag}(\X_L)\|_F^2\\
    & =  \frac{1}{\rho}\left(\rho\|\A\|_F^2\right)\left(\rho\|\X\|_F^2\right) &\text{by \eqref{property1}}\\
    & = \rho \|\A\|_F^2\|\X\|_F^2.
\end{align*}
After taking the square root on both sides, we obtain the desired result. Note that this result is consistent with that for third-order tensors with Fourier transform.
\end{proof}

\begin{definition}\label{bregman}
For any convex function $f: \R^{n_1\times n_2\times ...\times n_m}\rightarrow \R$, the \textbf{Bregman distance} between $\X$ and $\Y$ with respect to $f$ and $\X^*\in \partial f(\X)$ is defined as
\[D_{f,\X^*}(\X,\Y) = f(\Y)-f(\X)-\langle \X^*,\Y-\X\rangle.\]
\end{definition}

\begin{definition} The \textbf{convex conjugate} of $f$ at $\Z$ is defined as
\[f^*(\Z) = \sup_{\X}\{\langle\Z,\X\rangle - f(\X)\}.\]
\end{definition}
Due to the duality property $\langle \X,\X^*\rangle = f(\X)+f^*(\X^*)$ when $\X^*\in \partial f(\X)$ \cite[Theorem 23.5]{rockafellar2015convex}, the Bregman distance can be written as
\[D_{f,\X^*}(\X,\Y) = f(\Y)+f^*(\X^*)-\langle \X^*,\Y\rangle.\] In general, if $f$ is $\alpha$-strongly convex, then the Bregman distance satisfies
\begin{align}\label{bregmanprop}
    \frac{{\alpha}}{2}\|\X-\Y\|_F^2\leq D_{f,\X^*}(\X,\Y).
\end{align}

\begin{proposition}\label{Prop1} Suppose $\A\in\R^{1\times n_2 \times n_3\times ...\times n_m}$, $\B\in\R^{1\times l\times n_3\times ...\times n_m}$ and $f$ is an ${\alpha}$-strongly convex function defined on $\R^{n_2\times l\times n_3\times ...\times n_m}$. Given arbitrary $\overline{\Z}\in \R^{n_2\times l\times n_3\times ...\times n_m}$ and $\overline{\X}=\nabla f^*(\overline{\Z})$, let $\Z = \overline{\Z}+t\A^**_L\frac{\B-(\A*_L\overline{\X})}{||\A(i(k))||_F^2}$  and $\X = \nabla f^*(\Z).$ Then we have

\[D_{f,\Z}(\X,\mathcal{H})\leq D_{f,\overline{\Z}}(\overline{\X},\mathcal{H}) - \frac{t}{\|\A\|_F^2}\left(1-\frac{t \rho}{2{\alpha}}\right)\|\B-\A*_L\overline{\X}\|_F^2\]
for any $\mathcal{H}$ that satisfies $\A*_L\mathcal{H}= \B$.
\end{proposition}
\begin{proof}
    Let $\W = \A^**_L(\B-\A*_L\overline{\X})$ and $s = \frac{t}{\|\A\|_F^2}$. Then $\Z = \overline{Z}+s\W$ and by \eqref{convlemma1} we have
    \[\|\W\|_F = \|\A^**_L(\B-\A*_L\overline{\X})\|_F = \sqrt{\rho}\|\A\|_F\|\B-A*_L\overline{\X}\|_F.\]
By \eqref{property1} we have
\begin{align*}
    \langle \W,\mathcal{H}-\overline{\X}\rangle& = \langle  \A^**_L(\B-\A*_L\overline{\X}),\mathcal{H}-\overline{\X}\rangle\\
    &=  \langle  \A^**_L(\A*_L\mathcal{H}-\A*_L\overline{\X}),\mathcal{H}-\overline{\X}\rangle\\
    &=\langle  \A*_L\mathcal{H}-\A*_L\overline{\X}, \A*_L(\mathcal{H}-\overline{\X})\rangle\\
    &=\|\A*_L\mathcal{H}-\A*_L\overline{\X}\|_F^2\\
    &=\|\B-\A*_L\overline{\X}\|_F^2
\end{align*}
Then using these results the proof follows as in \cite{chen2021regularized}.
\end{proof}
Consider $f(\X)$ as defined in \eqref{eqn:feqn}. Since $f$ is 1-strongly convex, the proof of \eqref{thm:cycconv} follows directly from the proposition above when we take
    \[
        \A(i(k)) = \A,\quad
        \B(i(k)) = \B,\quad
        \Z^{(k)} = \overline{Z},\quad
        \Z^{(k+1)} = \Z.
    \] The proof of the convergence of $\{\X^{(k)}\}$ follows similarly as the proof in \cite[Theorem 3.3]{chen2021regularized}.

\section{Proof of Theorem \ref{thm:randconv}}\label{app:b}

\begin{definition}  A function $f:\R^{n_1\times n_2\times ...\times n_m}\rightarrow \R$ is \textbf{admissible} if the dual function $g_f$, as defined in \cite[Section 3.2]{chen2021regularized} is restricted strongly convex on any of $g_f$'s levels sets. Function $f$ is \textit{strongly admissible} if $g_f$ is restricted strongly convex on $\R^{n_1\times k\times n_3\times ...\times n_m}$.
\end{definition}

\begin{lemma}\label{lemma:nu}
 Assume $f$ is admissible, and let $\hat{\X}$ be the solution to \eqref{pth}. For $\Tilde{\X}$ and $\Tilde{\Z}\in \partial f(\Tilde{\X})\cap R(\A)$, there exists $\nu>0$ such that for all $\X$ and $\Z\in \partial f(\X)\cap R(\A)$ with $D_{f,\Z}(\X,\hat{\X})\leq D_{f,\Tilde{\Z}}(\Tilde{\X},\hat{\X})$, it holds that
\begin{align}\label{nulemma}
    D_{f,\Z}(\X,\hat{\X})\leq \frac{1}{\nu}\|\A*_L(\X-\hat{\X})\|_F^2.
\end{align}
\end{lemma}
\section{Proof of Theorem \ref{thm:blockconv}}\label{app:c}

 The key part of proving this theorem lies in the following bound where $M$ is the number of blocks.
    From Theorem \eqref{thm:cycconv} we have
    \begin{align*}f(\X^{(k)})-f(\X^{(k+1)})&\leq \langle \Z^{(k+1)},\X-\X^{(k+1)}\rangle -\langle \Z^{(k)},\X-\X^{(k)}\rangle \ \\
&- t\left(1-\frac{t\rho}{2}\right)\frac{\| A(\tau_k)*_L(\X^{(k)}-\X)\|_F^2}{\|\A(\tau_k)\|_F^2}
\end{align*}
We may rearrange and express this equivelently as
 \begin{align*}D_{f,\Z^{(k+1)}}(\X^{(k+1)},\hat{\X})&\leq D_{f,\Z^{(k)}}(\X^{(k)},\hat{\X})- t\left(1-\frac{t\rho}{2}\right)\frac{\| A(\tau_k)*_L(\X^{(k)}-\X)\|_F^2}{\|\A(\tau_k)\|_F^2}
\end{align*}
Now we want to analyze the expectation of $\frac{\| \A(\tau_k)*_L(\X^{(k)}-\X)\|_F^2}{\|\A(\tau_k)\|_F^2}$, where $\A(\tau_k) = \A(\tau_k,:,:,...,:)$. Then let $\mathbb{E}_c$ be the expectation conditioned on $\tau_0,...\tau_{k-1}$.

By the definition of expectation and Lemma~\ref{lemma:nu}, we have
\begin{align*}
   \mathbb{E}_c\left[\frac{\| \A(\tau_k)*_L(\X^{(k)}-\X)\|_F^2}{\|\A(\tau_k)\|_F^2}\right] &= \frac{1}{M}\sum_{j\in \tau_k}\frac{\| \A(j)*_L(\X^{(k)}-\X)\|_F^2}{\|\A(j)\|_F^2} \\
   & = \frac{1}{M}\frac{\| \A*_L(\X^{(k)}-\X)\|_F^2}{\|\A\|_F^2}\\
   & \geq  \frac{1}{M} \frac{\nu}{\|\A\|_F^2}D_{f,\Z^{(k)}}(\X^{(k)},\hat{\X}).
\end{align*} The last inequality follows from \eqref{nulemma}. Then taking the expectation leads to
\begin{align*}
    &\mathbb{E}\left[D_{f,\Z^{(k+1)}}(\X^{(k+1)},\hat{\X})\right]\\
    & \leq
    \mathbb{E}\left[D_{f,\Z^{(k)}}(\X^{(k)},\hat{\X})-t\left(1-\frac{t \rho }{2}\right)\frac{\| \A(\tau_k)*_L(\X^{(k)}-\X)\|_F^2}{\|\A(\tau_k)\|_F^2}\right]\\
    & = \mathbb{E}_{\tau_0,...,\tau_{k-1}}\mathbb{E}_c\left[D_{f,\Z^{(k)}}(\X^{(k)},\hat{\X})-t\left(1-\frac{t \rho }{2}\right)\frac{\| \A(\tau_k)*_L(\X^{(k)}-\X)\|_F^2}{\|\A(\tau_k)\|_F^2}\right]\\
    &\leq \mathbb{E}_{\tau_0,...,\tau_{k-1}}\left[D_{f,\Z^{(k)}}(\X^{(k)},\hat{\X})-t\left(1-\frac{t \rho }{2}\right)\frac{1}{M} \frac{\nu}{\|\A\|_F^2}D_{f,\Z^{(k)}}(\X^{(k)},\hat{\X})\right]\\
    &\leq \left(1-\frac{\nu t}{M\|\A\|_F^2}\left(1-\frac{t \rho }{2}\right)\right)\mathbb{E}\left[D_{f,\Z^{(k)}}(\X^{(k)},\hat{\X})\right].
\end{align*}
After applying the above inequality repeatedly, we can bound the expected Bregman distance after $k$ iterations in terms of the initial Bregman distance.
\[
\mathbb{E}\left[D_{f,\Z^{(k)}}(\X^{(k)},\hat{\X})\right]\leq \left(1-\frac{\nu t}{M\|\A\|_F^2}\left(1-\frac{t \rho }{2}\right)\right)^k\left[D_{f,\Z^{(0)}}(\X^{(0)},\hat{\X})\right].
\]
Then by \eqref{bregmanprop}, we have
\begin{align*}
    &\mathbb{E}\|\X^{(k)}-\hat{\X}\|_F^2\leq 2\mathbb{E}\left[D_{f,\Z^{(k)}}(\X^{(k)},\hat{\X})\right]\\
    &\leq 2 D_{f,\Z^{(0)}}(\X^{(0)},\hat{\X})\left(1-\frac{\nu t}{M\|\A\|_F^2}\left(1-\frac{t \rho }{2}\right)\right)^k\\
    &=2\left[f(\hat{\X})-f(\X^{(0)})-\langle \Z^{(0)},\hat{\X}-\X^{(0)}\rangle \right]\left(1-\frac{\nu t}{M\|\A\|_F^2}\left(1-\frac{t \rho }{2}\right)\right)^k.
\end{align*}
Since $f(\X)$, as defined in \eqref{eqn:feqn}, is $1$-strongly convex, by Definition \ref{def:stconv} we have
\begin{align*}
    \mathbb{E}\|\X^{(k)}-\hat{\X}\|_F^2
    &\leq  \|\hat{\X}-\X^{(0)}\|_F^2\left(1-\frac{\nu t}{M\|\A\|_F^2}(1-\frac{t \rho}{2})\right)^k.
\end{align*}
Note that this result relies on $M$, the number of blocks.

\bibliographystyle{unsrt}
\bibliography{references}

\begin{thebibliography}{10}

\bibitem{karczmarz1937angenaherte}
Stefan Karczmarz.
\newblock Angenaherte auflosung von systemen linearer glei-chungen.
\newblock {\em Bull. Int. Acad. Pol. Sic. Let., Cl. Sci. Math. Nat.}, pages
  355--357, 1937.

\bibitem{gordon1970algebraic}
Richard Gordon, Robert Bender, and Gabor~T Herman.
\newblock Algebraic reconstruction techniques ({ART}) for three-dimensional
  electron microscopy and {X}-ray photography.
\newblock {\em Journal of theoretical Biology}, 29(3):471--481, 1970.

\bibitem{popa2004kaczmarz}
Constantin Popa and Rafal Zdunek.
\newblock Kaczmarz extended algorithm for tomographic image reconstruction from
  limited-data.
\newblock {\em Mathematics and Computers in Simulation}, 65(6):579--598, 2004.

\bibitem{zhou2013tensor}
Hua Zhou, Lexin Li, and Hongtu Zhu.
\newblock Tensor regression with applications in neuroimaging data analysis.
\newblock {\em Journal of the American Statistical Association},
  108(502):540--552, 2013.

\bibitem{eggermont1981iterative}
Paulus Petrus~Bernardus Eggermont, Gabor~T Herman, and Arnold Lent.
\newblock Iterative algorithms for large partitioned linear systems, with
  applications to image reconstruction.
\newblock {\em Linear algebra and its applications}, 40:37--67, 1981.

\bibitem{chen2021regularized}
Xuemei Chen and Jing Qin.
\newblock Regularized {K}aczmarz algorithms for tensor recovery.
\newblock {\em SIAM Journal on Imaging Sciences}, 14(4):1439--1471, 2021.

\bibitem{strohmer2009randomized}
Thomas Strohmer and Roman Vershynin.
\newblock A randomized {K}aczmarz algorithm with exponential convergence.
\newblock {\em Journal of Fourier Analysis and Applications}, 15(2):262, 2009.

\bibitem{needell2014paved}
Deanna Needell and Joel~A Tropp.
\newblock Paved with good intentions: analysis of a randomized block {K}aczmarz
  method.
\newblock {\em Linear Algebra and its Applications}, 441:199--221, 2014.

\bibitem{needell2015randomized}
Deanna Needell, Ran Zhao, and Anastasios Zouzias.
\newblock Randomized block {K}aczmarz method with projection for solving least
  squares.
\newblock {\em Linear Algebra and its Applications}, 484:322--343, 2015.

\bibitem{burger2006regularizing}
Martin Burger and Barbara Kaltenbacher.
\newblock Regularizing newton--kaczmarz methods for nonlinear ill-posed
  problems.
\newblock {\em SIAM Journal on Numerical Analysis}, 44(1):153--182, 2006.

\bibitem{lorenz2014linearized}
Dirk~A Lorenz, Frank Schopfer, and Stephan Wenger.
\newblock The linearized bregman method via split feasibility problems:
  analysis and generalizations.
\newblock {\em SIAM Journal on Imaging Sciences}, 7(2):1237--1262, 2014.

\bibitem{wang2023solving}
Xuezhong Wang, Maolin Che, Changxin Mo, and Yimin Wei.
\newblock Solving the system of nonsingular tensor equations via randomized
  {K}aczmarz-like method.
\newblock {\em Journal of Computational and Applied Mathematics}, 421:114856,
  2023.

\bibitem{ma2022randomized}
Anna Ma and Denali Molitor.
\newblock Randomized {K}aczmarz for tensor linear systems.
\newblock {\em BIT Numerical Mathematics}, 62(1):171--194, 2022.

\bibitem{tibshirani1996regression}
Robert Tibshirani.
\newblock Regression shrinkage and selection via the lasso.
\newblock {\em Journal of the Royal Statistical Society Series B: Statistical
  Methodology}, 58(1):267--288, 1996.

\bibitem{candes2008enhancing}
Emmanuel~J Candes, Michael~B Wakin, and Stephen~P Boyd.
\newblock Enhancing sparsity by reweighted $\ell_1$ minimization.
\newblock {\em Journal of Fourier analysis and applications}, 14:877--905,
  2008.

\bibitem{henneberger2023log}
Katherine Henneberger and Jing Qin.
\newblock Log-sum regularized {K}aczmarz algorithms for high-order tensor
  recovery.
\newblock {\em Preprint arXiv: 2311.00783}, 2023.

\bibitem{prater2023constructive}
Ashley Prater-Bennette, Lixin Shen, and Erin~E Tripp.
\newblock A constructive approach for computing the proximity operator of the
  $p$-th power of the $\ell_1$ norm.
\newblock {\em Applied and Computational Harmonic Analysis}, 67:101572, 2023.

\bibitem{qin2022low}
Wenjin Qin, Hailin Wang, Feng Zhang, Jianjun Wang, Xin Luo, and Tingwen Huang.
\newblock Low-rank high-order tensor completion with applications in visual
  data.
\newblock {\em IEEE Transactions on Image Processing}, 31:2433--2448, 2022.

\bibitem{kilmer2019tensor}
Misha Kilmer, Lior Horesh, Haim Avron, and Elizabeth Newman.
\newblock Tensor-tensor products for optimal representation and compression.
\newblock {\em arXiv preprint arXiv:2001.00046}, 2019.

\bibitem{lu2018tensor}
Canyi Lu.
\newblock Tensor-tensor product toolbox.
\newblock {\em arXiv preprint arXiv:1806.07247}, 2018.

\bibitem{parikh2014proximal}
Neal Parikh, Stephen Boyd, et~al.
\newblock Proximal algorithms.
\newblock {\em Foundations and trends{\textregistered} in Optimization},
  1(3):127--239, 2014.

\bibitem{rockafellar2015convex}
Ralph~Tyrell Rockafellar.
\newblock {\em Convex Analysis:(PMS-28)}.
\newblock Princeton university press, 2015.

\bibitem{prater2022proximity}
Ashley Prater-Bennette, Lixin Shen, and Erin~E Tripp.
\newblock The proximity operator of the log-sum penalty.
\newblock {\em Journal of Scientific Computing}, 93(3):67, 2022.

\bibitem{boyd2004convex}
Stephen~P Boyd and Lieven Vandenberghe.
\newblock {\em Convex optimization}.
\newblock Cambridge university press, 2004.

\bibitem{necoara2019faster}
Ion Necoara.
\newblock {Faster randomized block {K}aczmarz algorithms}.
\newblock {\em SIAM Journal on Matrix Analysis and Applications},
  40(4):1425--1452, 2019.

\bibitem{barzilai1988two}
Jonathan Barzilai and Jonathan~M Borwein.
\newblock {Two-point step size gradient methods}.
\newblock {\em IMA journal of numerical analysis}, 8(1):141--148, 1988.

\bibitem{nesterov1983method}
Yurii Nesterov.
\newblock {A method for unconstrained convex minimization problem with the rate
  of convergence $O(1/k^2)$}.
\newblock In {\em Dokl. Akad. Nauk. SSSR}, volume 269, page 543, 1983.

\bibitem{wang2004image}
Zhou Wang, Alan~C Bovik, Hamid~R Sheikh, and Eero~P Simoncelli.
\newblock Image quality assessment: from error visibility to structural
  similarity.
\newblock {\em IEEE transactions on image processing}, 13(4):600--612, 2004.

\bibitem{zhang2011fsim}
Lin Zhang, Lei Zhang, Xuanqin Mou, and David Zhang.
\newblock Fsim: A feature similarity index for image quality assessment.
\newblock {\em IEEE transactions on Image Processing}, 20(8):2378--2386, 2011.

\bibitem{georghiades2001few}
Athinodoros~S. Georghiades, Peter~N. Belhumeur, and David~J. Kriegman.
\newblock From few to many: Illumination cone models for face recognition under
  variable lighting and pose.
\newblock {\em IEEE transactions on pattern analysis and machine intelligence},
  23(6):643--660, 2001.

\bibitem{wang2023guaranteed}
Hailin Wang, Jiangjun Peng, Wenjin Qin, Jianjun Wang, and Deyu Meng.
\newblock Guaranteed tensor recovery fused low-rankness and smoothness.
\newblock {\em IEEE Transactions on Pattern Analysis and Machine Intelligence},
  2023.

\bibitem{gilles2011bregman}
Jerome Gilles.
\newblock {The Bregman Cookbook}.
\newblock {\em Bregman Cookbook}, pages 10--30, 2011.

\bibitem{zhang2021plug}
Kai Zhang, Yawei Li, Wangmeng Zuo, Lei Zhang, Luc Van~Gool, and Radu Timofte.
\newblock Plug-and-play image restoration with deep denoiser prior.
\newblock {\em IEEE Transactions on Pattern Analysis and Machine Intelligence},
  44(10):6360--6376, 2021.

\bibitem{richardson1972bayesian}
William~Hadley Richardson.
\newblock Bayesian-based iterative method of image restoration.
\newblock {\em JoSA}, 62(1):55--59, 1972.

\bibitem{lucy1974iterative}
Leon~B Lucy.
\newblock An iterative technique for the rectification of observed
  distributions.
\newblock {\em Astronomical Journal, Vol. 79, p. 745 (1974)}, 79:745, 1974.

\end{thebibliography}
\end{document}